\documentclass[11pt]{amsart}
\usepackage{amsthm}

\usepackage[all]{xy}
\usepackage{amssymb}
\usepackage{enumerate}
\usepackage{mathrsfs}
\usepackage{graphicx}
\usepackage{caption}
\usepackage{subcaption}
\usepackage{float}
\usepackage{epigraph}

\numberwithin{equation}{section}

\newtheorem{thm}{Theorem}[section]

\newtheorem{lem}[thm]{Lemma}

\newtheorem{rem}{Remark}[section]
\newtheorem{example}[thm]{Example}

\newtheorem{defin}[thm]{Definition}

\newcommand{\eq}[1]{(\ref{#1})}

\renewcommand{\Re}{\operatorname{\rm Re}}
\renewcommand{\Im}{\operatorname{\rm Im}}

\newcommand{\beqast}{\begin{eqnarray*}}
\newcommand{\eqast}{\end{eqnarray*}}
\newcommand{\beqa}{\begin{eqnarray}}
\newcommand{\eqa}{\end{eqnarray}}

\newcommand{\bbe}{\begin{equation}}
\newcommand{\ee}{\end{equation}}

\renewcommand{\Re}{\operatorname{\rm Re}}
\renewcommand{\Im}{\operatorname{\rm Im}}

\newcommand{\bC}{{\mathbb C}}
\newcommand{\bE}{{\mathbb E}}
\newcommand{\bN}{{\mathbb N}}
\newcommand{\bP}{{\mathbb P}}

\newcommand{\bR}{{\mathbb R}}

\newcommand{\bZ}{{\mathbb Z}}

\newcommand{\cF}{{\mathcal F}}

\newcommand{\cE}{{\mathcal E}}
\newcommand{\cG}{{\mathcal G}}

\newcommand{\cL}{{\mathcal L}}

\newcommand{\cC}{{\mathcal C}}
\newcommand{\cQ}{{\mathcal Q}}
\newcommand{\cU}{{\mathcal U}}

\newcommand{\barX}{{\bar X}}
\newcommand{\uX}{{\underline X}}

\newcommand{\cEq}{{\mathcal E_q}}
\newcommand{\cEpq}{{\mathcal E^+_q}}
\newcommand{\cEmq}{{\mathcal E^-_q}}

\newcommand{\phipq}{{\phi^+_q}}
\newcommand{\phimq}{{\phi^-_q}}

\newcommand{\tV}{{\tilde V}}

\newcommand{\Om}{{\Omega}}

\newcommand{\De}{\Delta}
\newcommand{\de}{\delta}
\newcommand{\eps}{\epsilon}
\newcommand{\ka}{\kappa}

\newcommand{\lp}{\lambda_+}
\newcommand{\lm}{\lambda_-}

\newcommand{\mum}{\mu_-}
\newcommand{\mup}{\mu_+}
\newcommand{\mumpr}{\mu'_-}
\newcommand{\muppr}{\mu'_+}

\newcommand{\num}{\nu_-}
\newcommand{\nup}{\nu_+}

\newcommand{\sg}{\sigma}

\newcommand{\om}{\omega}
\newcommand{\omm}{\om_-}
\newcommand{\omp}{\om_+}

\newcommand{\ze}{\zeta}

\newcommand{\ga}{\gamma}
\newcommand{\gap}{\gamma_+}
\newcommand{\gam}{\gamma_-}
\newcommand{\gappr}{\gamma'_+}
\newcommand{\gampr}{\gamma'_-}

\newcommand{\Ga}{\Gamma}

\newcommand{\dd}{\partial}

\newcommand{\bfo}{{\bf 1}}

\begin{document}

\title[joint pdf of a L\'evy process, its extremum, and hitting time of the extremum]
{Efficient evaluation  of joint pdf of a L\'evy process, its extremum, and hitting time of the extremum}
\author[
Svetlana Boyarchenko and
Sergei Levendorski\u{i}]
{
Svetlana Boyarchenko and
Sergei Levendorski\u{i}}

\begin{abstract}
For L\'evy processes with exponentially decaying tails of the L\'evy density, we derive integral representations for the joint cpdf 
$V$
of $(X_T, \barX_T,\tau_T)$
(the process, its supremum evaluated at $T<+\infty$, and the first time at which $X$ attains its supremum).
The first representation is a Riemann-Stieltjes integral in terms of the (cumulative) probability distribution of the supremum process and 
joint probability distribution function of the process and its supremum process. The integral is evaluated using a combination an analog of the
 trapezoid rule. The second representation is amenable to more accurate albeit slower calculations. We calculate explicitly the Laplace-Fourier transform of $V$ w.r.t.  all arguments,
 apply the inverse transforms, and reduce the problem to evaluation of the sum  of 5D integrals.  
The integrals can be evaluated using the summation by parts in the infinite trapezoid rule and  simplified trapezoid rule; the inverse
Laplace transforms can be calculated using the  Gaver-Wynn-Rho algorithm.
Under additional conditions on the domain of analyticity of the characteristic exponent,  
the speed of calculations is greatly increased using the conformal deformation technique.
For processes of infinite variation, the program in Matlab running on a Mac with moderate characteristics achieves the precision better than E-05 in 
 a fraction of a second;  the precision better than  E-10 is achievable in dozens of seconds.
 As the order of the process (the analog of the Blumenthal-Getoor index) decreases, the CPU time increases, and
 the best accuracy achievable with double precision arithmetic decreases.

\end{abstract}

\thanks{
\emph{S.B.:} Department of Economics, The
University of Texas at Austin, 2225 Speedway Stop C3100, Austin,
TX 78712--0301, {\tt sboyarch@utexas.edu} \\
\emph{S.L.:}
Calico Science Consulting. Austin, TX.
 Email address: {\tt
levendorskii@gmail.com}}

\maketitle

\noindent
{\sc Key words:} L\'evy process, extrema of a L\'evy process,  barrier options, Wiener-Hopf factorization, Fourier transform, Laplace transform, 
 Gaver-Wynn Rho algorithm, sinh-acceleration, SINH-regular processes, Stieltjes-L\'evy processes

\noindent
{\sc MSC2020 codes:} 60-08,42A38,42B10,44A10,65R10,65G51,91G20,91G60

\tableofcontents

\section{Introduction}\label{s:intro}
Let $X$ be a one-dimensional L\'evy process on the filtered probability space $(\Om, \cF, \{\cF_t\}_{t\ge 0}, \bP)$
satisfying the usual conditions, and let $\bE$ be the expectation operator under $\bP$. 
Let  $\barX_t=\sup_{0\le s\le t}X_s$ and $\uX_t=\inf_{0\le s\le t}X_s$ be the supremum
and infimum processes (defined path-wise, a.s.); $X_0=\barX_0=\uX_0=0$. Let $\tau_T$ be the first time at which $X$ attains its supremum. The joint probability distribution $V(a_1,a_2;T,t):=\bP[X_T\le a_1,\barX_T\le a_2, \tau_T\le t]$, where $a_1\le a_2$, $a_2>0$ and $0<t\le T$,
of the triplet $\chi=(X_T, \barX_T, \tau_T)$ is an important object in insurance mathematics,
structural credit risk models, mathematical finance, buffer size in queuing theory and the prediction of the ultimate supremum and its time in optimal stopping. As it stated in \cite{MijatovicGeomConvSimul21}, for a general L\'evy process,  analytical calculations are extremely challenging, which lead to the development of numerous approximate methods, mostly, Monte Carlo and multi-level Monte Carlo. See \cite{MijatovicGeomConvSimul21} for a novel advanced Monte Carlo method and review of the related literature.  

 In the paper, we suggest fairly accurate and fast analytical methods for the evaluation of $V(a_1,a_2;T,t)$, applicable to essentially all 
 popular classes of L\'evy processes. The first method ({\em DISC-method}) is very simple and relies on efficient procedures for the evaluation
 of $V_{f.t.d.}(h,t):=\bP[\barX_t\ge h]$ and $V_{joint}(a_1,h; T):=\bP[X_T\le a_1, \barX_T\le h]$, $0<h$, $a_1\le h$, the distributions of the supremum of the L\'evy process and its supremum. We discretize the Riemann-Stieltjes integral
  \bbe\label{eq:main1}
V(a_1,a_2;T,t)=\int_{0}^{a_2}V_{f.t.d.}(h,t)\,d_hV_{joint}(\min\{a_1,h\},h; T)
\ee
  using the trapezoid type 
formula for the Riemann-Stieltjes integrals,
introduced in 
\cite{Dragomir11} 
\cite{Dragomir11}:
\beqa\label{RiemannStTrap}
\int_{0}^{a_2} V_{f.t.d.}(h,t)dV_{joint}(\min\{a_1,h\},h,T)
\approx\sum_{j=0}^{M-1}\frac{V_{f.t.d.}(h_j,t)+V_{f.t.d.}(h_{j+1},t)}{2}\\\nonumber
\times (V_{joint}(\min\{a_1,h_{j+1}\},h_{j+1},T)-V_{joint}(\min\{a_1,h_j\},h_{j},T)),
\eqa
where $0=h_0<h_1<\cdots<h_M=a_2$.
One can use different quadrature rules for {R}iemann-{S}tieltjes integrals. See, e.g. \cite{Alomari19}.
If the values of the functions $V_{f.t.d.}(h,t)$ and $V_{joint}(a_1,h;T)$ are calculated using GWR algorithm for the Laplace
inversion (resp., sinh-acceleration in the Bromwich integral), we use the name DISC-GWR (resp., DISC-SINH) method.
Using GWR method and double precision arithmetic, the individual terms on the RHS of 
\eq{RiemannStTrap} are difficult to calculate with the accuracy better than E-08, hence, with any choice of the grid,
 in \eq{RiemannStTrap}, 
it is essentially impossible to achieve the accuracy better than E-08.
If DISC-SINH is used, the individual terms can be calculated fairly fast with the accuracy E-12 (see \cite{EfficientLevyExtremum} for examples), hence, better accuracy can be achieved, at a larger CPU cost.

Simple algorithms developed in \cite{Contrarian,EfficientLevyExtremum} allow one to calculate 
  $V_{joint}(\min\{a_1,h\},h,T)$ 
and $V_{f.t.d.}(h,t)$ with almost machine precision fairly fast unless the order $\nu$ of the process 
(the generalization of the Blumenthal-Getoor index) is close to 0. One algorithm uses the Gaver-Wynn-Rho algorithm 
for the Laplace inversion, the other one is based on the sinh-deformation of the contour of integration in the Bromwich integral
and the corresponding sinh-change of variables. Thus, we have two versions: {\em DISC-GWR} and 
{\em DISC-SINH} methods, the former being faster and the latter more accurate. In the majority of popular classes
of L\'evy models, the value functions of barrier options  are very irregular at the boundary unless a sizable Brownian motion component is present (see Section \ref{ss:regularity}), hence, in a vicinity of $h=0$ and vicinity of $h=a_1$ if $a_1>0$,
the discretization error decreases very slowly, and practical sufficiently accurate error bounds are extremely difficult to derive.
We estimate  errors of DISC-methods using {\em SINH method},  which is applicable if  $\dd_a\dd_hV_{joint}(a,h; T)$ is integrable on $\{a<h, h>0\}$, and the Laplace-Fourier transform of  $\dd_a\dd_hV_{joint}(a,h; T)$ can be efficiently calculated. This is the case for the majority of popular L\'evy processes
 \cite{EfficientLevyExtremum} bar stable L\'evy processes\footnote{The method admits a modification to the case of stable L\'evy processes in the same vein as the method of
\cite{EfficientLevyExtremum} is modified in \cite{EfficientStableLevyExtremum}.} and L\'evy processes of finite variation, with non-zero drfit. 
Following \cite{EfficientLevyExtremum}, for $a_1\le h$, $h>0$, introduce $V_2(a_1,h; T)=V_{joint}(a_1,h; T)-\bP[X_T\le a_1]$. Using explicit representations of the Laplace-Fourier transform of the functions on the RHS of the equation 
 \beqa\label{eq:main2}
V(a_1,a_2;T,t)&=&\bfo_{a_1\le 0}\int_{0}^{a_2}dh\,V_{f.t.d.}(h,t)\,\dd_hV_2(a_1,h; T)\\\nonumber
&&+\bfo_{a_1>0}\left\{\int_{0}^{a_2}dh\,V_{f.t.d.}(h,t)\,\dd_hV_2(0,h; T)\right.
\\\nonumber
&&\left.+\int_{0}^{a_1}da\int_{a}^{a_2}dh\,V_{f.t.d.}(h,t)\,\dd_a\dd_hV_2(a,h; T)\right\},
\eqa 
we change the variables, calculate the integrals w.r.t. $a, h$ explicitly, and express $V(a_1,a_2;T,t)$ as a sum
of quintuple integrals (two Laplace inversions and three Fourier inversions in each integral).
 The resulting Laplace-Fourier inversion formula is justified
in the sense of generalized functions for all popular L\'evy processes; the justification in the classical sense and efficient evaluation of the resulting integrals are possible
for a wide subclass of SINH-regular processes
 introduced in \cite{SINHregular}. Conformal deformations of the lines of integration with the subsequent corresponding change
 of variables of the form $\xi=i\om_1+b\sinh(i\om+y)$ and $q=\sg+ib_\ell\sinh(i\om_\ell+y_\ell)$ in the Fourier and Bromwich integrals, respectively, and application of the simplified trapezoid rule allow one to evaluate the integrals accurately.
The errors of SINH method are fairly easy to control. The complexity of the method is of the order of
$(E\ln E)^5$, where $E=\ln(1/\eps)$. For processes of order $\nu\ge 1$, 
the error tolerance  of the order of $10^{-10}$
 can be satisfied using double precision arithmetic; the best accuracy achievable using double precision arithmetic decreases as $\nu\to 0$.  The accuracy of SINH-method is significantly higher than that of DISC methods; the CPU cost is higher
 as well. 
\begin{figure}
\begin{tabular}{cc}

 \begin{subfigure}[h]{0.45\textwidth}

 \centering
    \includegraphics[width=0.9\textwidth,height=0.4\textheight]{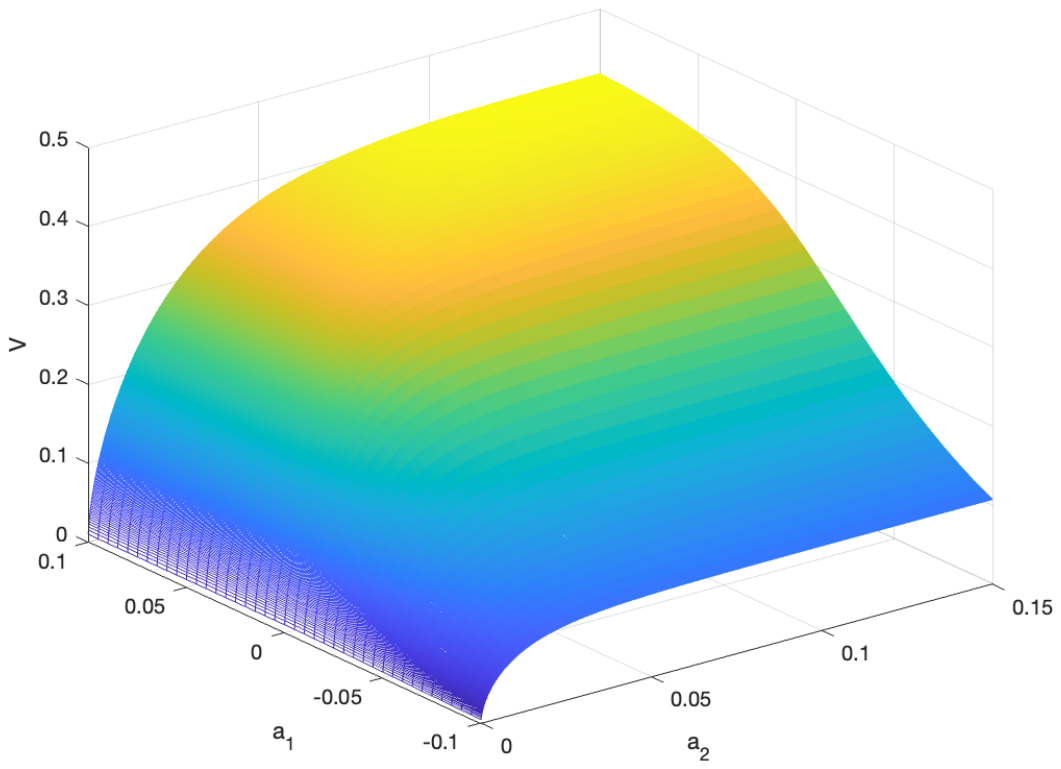}
    \caption{$\nu=1.2$}\label{nu1.2}
\end{subfigure}
&
\begin{subfigure}[h]{0.45\textwidth}
\centering
    \includegraphics[width=0.9\textwidth,height=0.4\textheight]{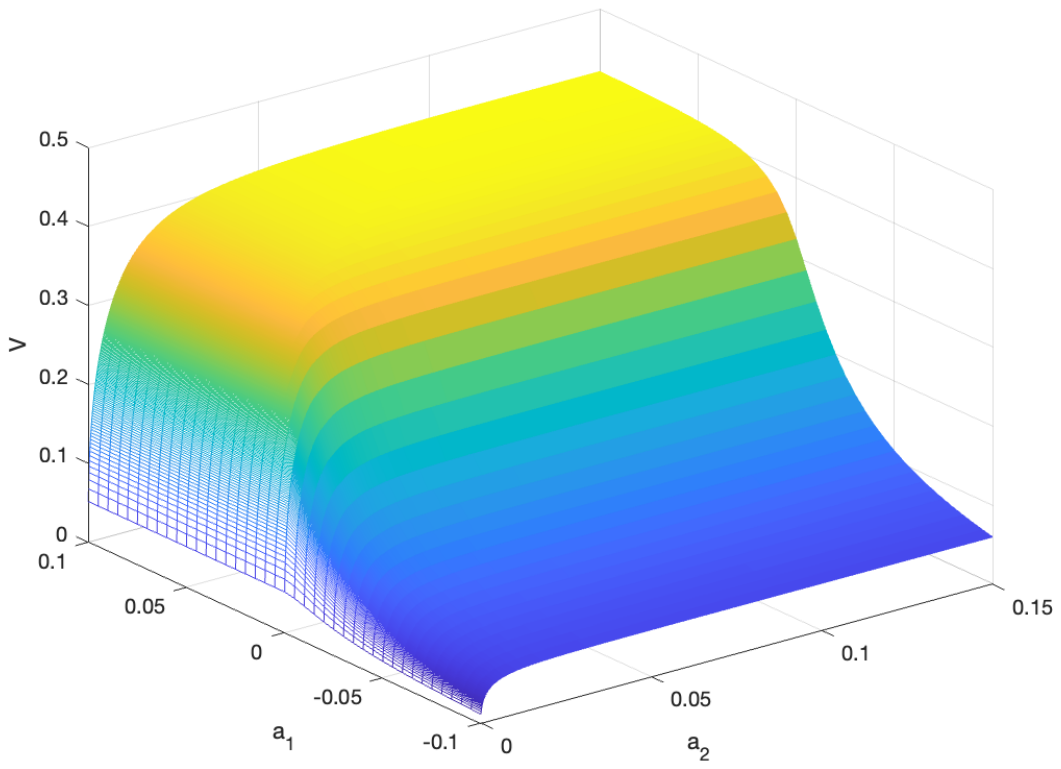}\caption{$\nu=0.8$.} \label{nu0.8}
\end{subfigure}
\\ \\

\begin{subfigure}[h]{0.45\textwidth}
\centering
    \includegraphics[width=0.9\textwidth,height=0.4\textheight]{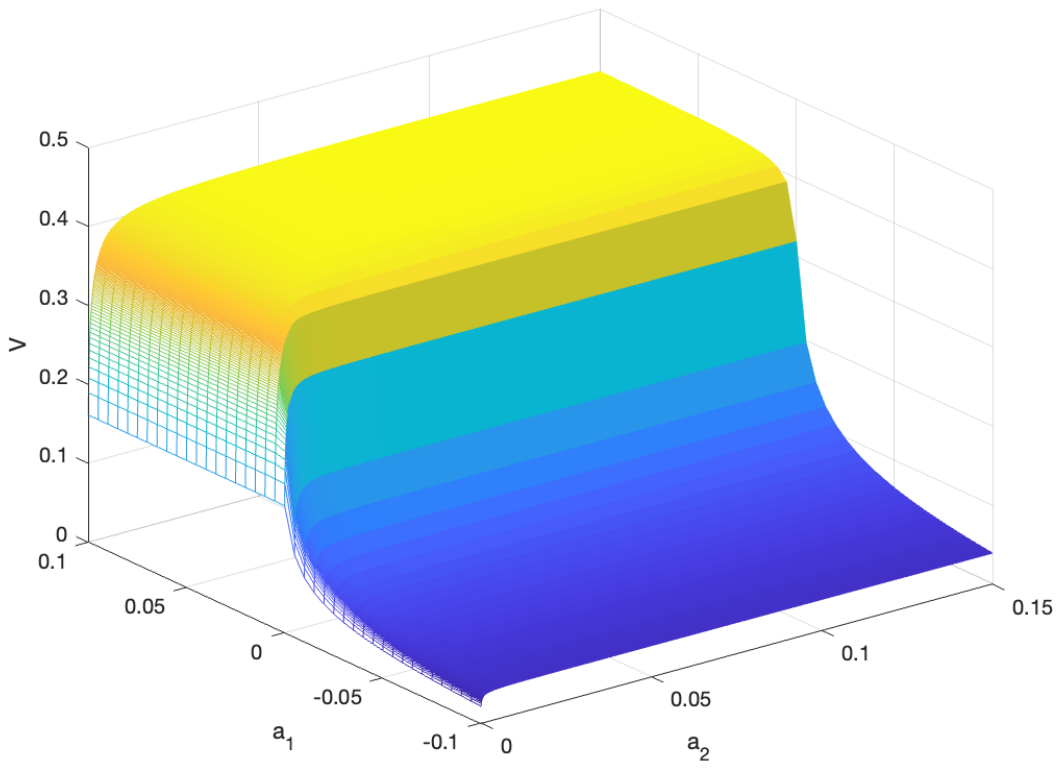} \caption{$\nu=0.5$}\label{nu05} \end{subfigure}
&
\begin{subfigure}{0.45\textwidth}
   \centering
    \includegraphics[width=0.9\textwidth,height=0.4\textheight]{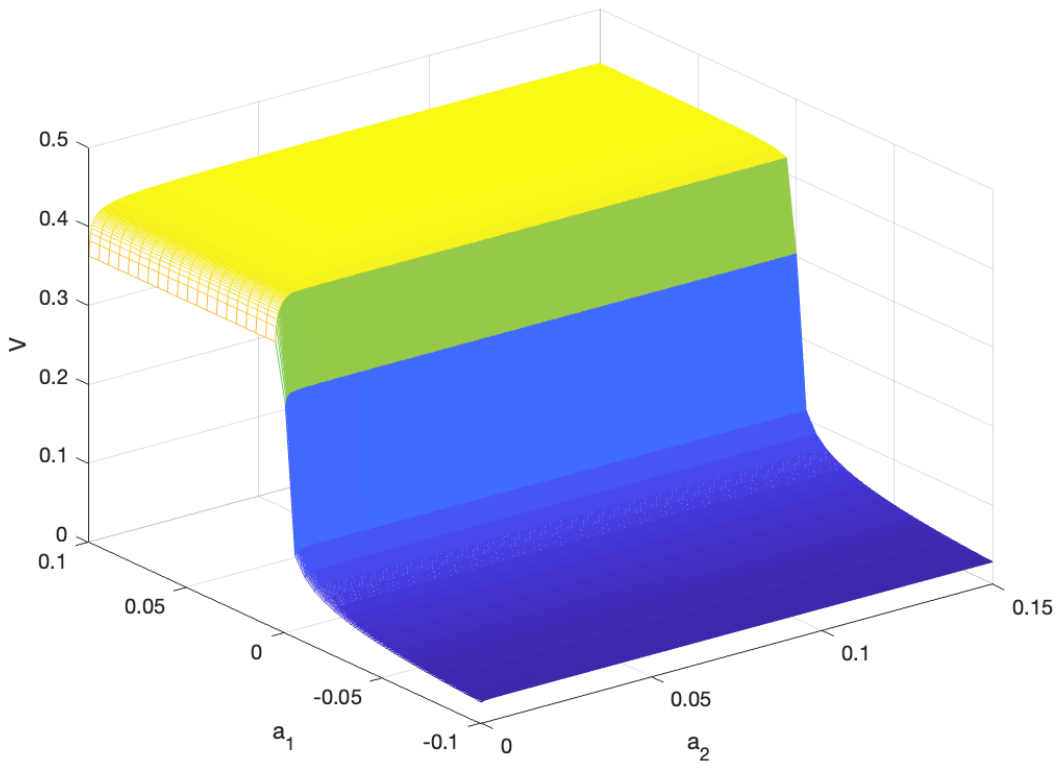}
    \caption{$\nu=0.3$ }\label{nu03}
\end{subfigure}
\end{tabular}
\caption{Joint cpdf $V(a_1,a_2;T,t):=\bP[X_T\le a_1,\barX_T\le a_2, \tau_T\le t]$. $X$: KoBoL, $\nu=0.3,0.5,0.8,1.2, \mu=0, \lp=1, \lm=-2$,
$\mu=0$,
the second instantaneous moment $m_2=0.1$ fixes the parameter $c$ of KoBoL. $T=0.25, t=0.1$. Grids: $\vec{a_1}=0.005*(-20:1:20),
\vec{a_2}=6.125*10^{-5}*(1:1:2400)$; total number of points in $(a_1,a_2)$ space is 98,400. For each $\nu$, the total CPU time  (one run) is of the order of 110-120 sec.
Errors are in the range: (A) $[2*10^{-7}, 2*10^{-6}]$; (B) $[10^{-6}, 10^{-5}]$; (C) $[10^{-5}, 10^{-4}]$; (D) $[0.003,0.02]$.}
\label{meshes}
\end{figure}
\begin{figure}
\scalebox{0.5}
{\includegraphics{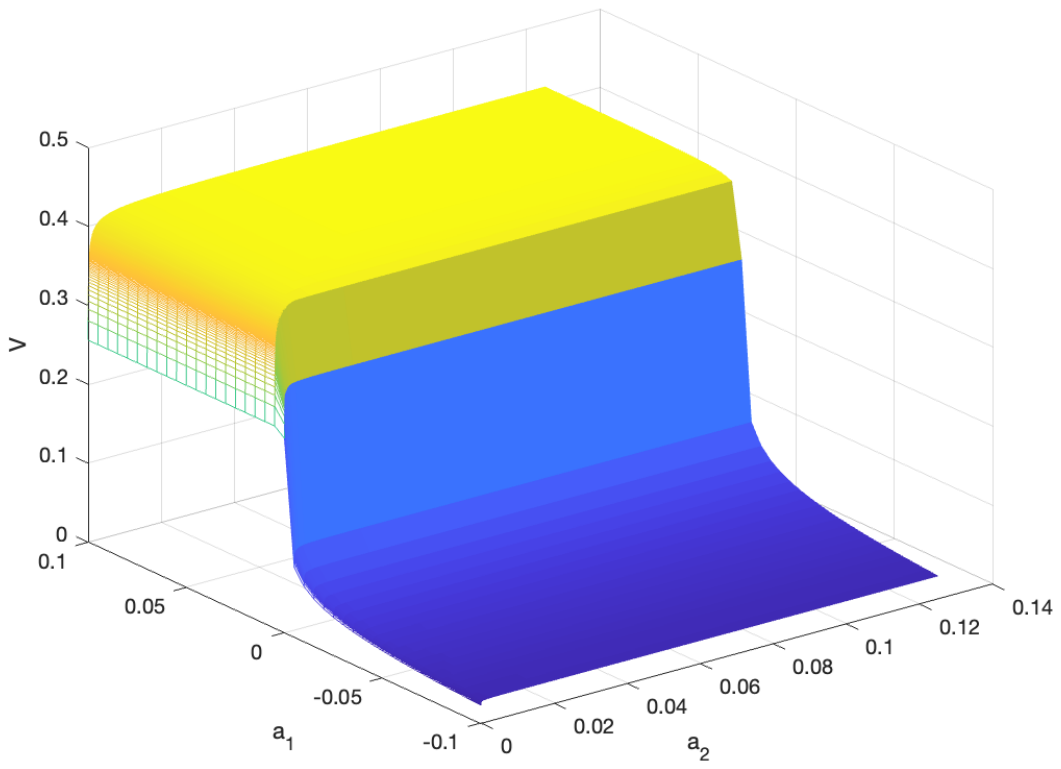}}
\caption{\small Joint cpdf $V(a_1,a_2;T,t):=\bP[X_T\le a_1,\barX_T\le a_2, \tau_T\le t]$. $X$: KoBoL, $\nu=0.3, \mu=0, \lp=1, \lm=-2$,
$\mu=0$,
the second instantaneous moment $m_2=0.1$ fixes the parameter $c$ of KoBoL. $T=0.25, t=0.1$. Grids: $\vec{a_1}=0.005*(-20:1:20),
\vec{a_2}=1.5625*10^{-5}*(1:1:9600)$; total number of points in $(a_1,a_2)$ space is 393,600. CPU time is 839,832 sec}
\label{meshnu0.3}
\end{figure} 
In an example shown in Fig.~\ref{meshes}, DISC-GWR is used with the uniform grid; the step is $6.125\cdot 10^{-5}$, 
and the total CPU time for calculation of values at 98,400 points  is of the order of 110-120 sec. for all 4 cases $\nu=1.2,0.8,0.5,0.3$.\footnote{The calculations in the paper
were performed in MATLAB R2023b-academic use, on a MacPro Chip Apple M1 Max Pro chip
with 10-core CPU, 24-core GPU, 16-core Neural Engine 32GB unified memory,
1TB SSD storage.
The  CPU times shown can be significantly improved using parallelized calculations.} 
The accuracy decreases as $\nu$ decreases, and for $\nu=0.2$ (not shown), the best accuracy achievable using double precision
arithmetic is several percent and more. Fig.~\ref{meshes} illustrates the reason: as $\nu$ decreases,
the discretization error increases because the regularity of
the probability distribution decreases. In all 4 cases, the value of $V$ at $a_2=0$ is zero but even at a distance
$6.125*10^{-5}$ from 0, the value is non-negligible and increases as $\nu$ decreases; the derivative w.r.t. $a_2$ tends to infinity
as $a_2\downarrow 0$ in all 4 cases. Fig.~\ref{meshnu0.3} demonstrates that if $\nu=0.3$, the cpdf is not small even
at $a_2=1.5625*10^{-5}$.
The accuracy of SINH-method decreases with $\nu$ as well but the accuracy of the order of
$10^{-3}$ is achievable even in the case $\nu=0.2$ because all calculations are in the dual space.

  The rest of the paper is organized as follows.  In Sect. \ref{s:prel}, we recall the definition of SINH-regular processes
and SL-regular processes, formulas for the Laplace transforms of $V_{f.t.d.}(h,t)$, $V_2(a_1,h; T)$ in terms of the expected present value operators, and various formulas for the Wiener-Hopf factors. We also study the regularity
of the Wiener-Hopf factors. In  Sect. \ref{s: basic ingredients and regularity}, we list integral representations for $V_{f.t.d.}(h,t)$, $V_2(a_1,h; T)$ and 
$\dd_hV_2(a_1,h; T)$.
 In the same section, we outline other popular methods for pricing barrier options which can be used
instead of the method of \cite{single} and  explain the difficulties that these methods face. 
Explicit formulas of the SINH-method 
are derived
in Sect. \ref{s:main}. 
The algorithms and numerical examples are in Sect. \ref{s:numer}. 
In Section \ref{s:concl}, we  summarise the results of the paper,
In the appendix, we recall basic formulas and properties of the infinite trapezoid rule, summation by parts and GWR algorithm.

\section{Classes of processes and formulas for the Wiener-Hopf factors}\label{s:prel}

\subsection{General classes of L\'evy processes amenable to efficient calculations}\label{ss:gen_eff_Levy}
For $\nu=0+$ (resp., $\nu=1+$), set $|\xi|^\nu=\ln|\xi|$ (resp., $|\xi|^\nu=|\xi|\ln|\xi|$), and introduce the following complete ordering in
the set $\{0+,1+\}\cup (0,2]$: the usual ordering in $(0,2]$; $\forall\ \nu>0, 0+<\nu$; $1<1+$; $\forall\ \nu>1, 1+<\nu$.
We use  the notation  $S_{(\mum,\mup)}=\{\xi\ |\ \Im\xi\in (\mum,\mup)\}$,
 $\cC_{\gam,\gap}=\{e^{i\varphi}\rho\ |\ \rho> 0, \varphi\in (\gam,\gap)\cup (\pi-\gap,\pi-\gam)\}$, 
 $\cC_{\ga}=\{e^{i\varphi}\rho\ |\ \rho> 0, \varphi\in (-\ga,\ga)\}$.
 \begin{defin}\label{def:SINH_reg_proc_1D0}(\cite[Defin. 2.1]{EfficientAmenable})
 We say that $X$ is a SINH-regular L\'evy process  (on $\bR$) of   order
 $\nu$ and type $((\mum,\mup);\cC; \cC_+)$
 iff
the following conditions are satisfied:
\begin{enumerate}[(i)]
\item
$\nu\in\{0+,1+\}\cup (0,2]$ and $\mum<0< \mup$;
\item
$\cC=\cC_{\gam,\gap}, \cC_+=\cC_{\gampr,\gappr}$, where $\gam\le \gampr<0<\gappr\le\gap$;
\item
the characteristic exponent $\psi$ of $X$ admits the representation
\bbe\label{eq:reprpsi}
\psi(\xi)=-i\mu\xi+\psi^0(\xi),
\ee
where $\mu\in\bR$, and 
$\psi^0$ admits analytic continuation to $i(\mum,\mup)+ (\cC\cup\{0\})$;
\item
for any $\varphi\in (\gam,\gap)$, there exists $c_\infty(\varphi)\in \bC\setminus (-\infty,0]$ s.t.
\begin{equation}\label{asympsisRLPE}
\psi^0(\rho e^{i\varphi})\sim  c_\infty(\varphi)\rho^\nu, \quad \rho\to+\infty;
\end{equation}
\item
the function $(\gam,\gap)\ni \varphi\mapsto c_\infty(\varphi)\in \bC$ is continuous;
\item
for any $\varphi\in (\gampr, \gappr)$, $\Re c_\infty(\varphi)>0$.
\end{enumerate}
\end{defin}
\begin{example}\label{ex:KoBoL}{\rm  A generic process of Koponen's family was constructed in  \cite{genBS,KoBoL}  as  pure jump processes with the L\'evy measure of the form
\begin{equation}\label{KBLmeqdifnu}
F(dx)=c_+e^{\lm x}x^{-\nu_+-1}\bfo_{(0,+\infty)}(x)dx+
 c_-e^{\lp x}|x|^{-\nu_--1}\bfo_{(-\infty,0)}(x)dx,
\end{equation}
where $c_\pm>0, \nu_\pm\in [0,2), \lm<0<\lp$.  If $\nu_\pm\in (0,2), \nu_\pm\neq 1$,
\bbe\label{KBLnupnumneq01}
\psi^0(\xi)=c_+\Ga(-\nu_+)((-\lm)^{\nu_+}-(-\lm-i\xi)^{\nu_+})+c_-\Ga(-\nu_-)(\lp^{\nu_-}-(\lp+i\xi)^{\nu_-}).
\ee
 Note that a specialization
 $\nu_\pm=\nu\neq 1$, $c=c_\pm>0$, of KoBoL used in a series of numerical examples in \cite{genBS} was named CGMY model in \cite{CGMY} (and the labels were changed:
 letters $C,G,M,Y$ replace the parameters $c,\nu,\lm,\lp$ of KoBoL):
 \bbe\label{KBLnuneq01}
 \psi^0(\xi)= c\Ga(-\nu)[(-\lm)^{\nu}-(-\lm- i\xi)^\nu+\lp^\nu-(\lp+ i\xi)^\nu].
\ee
Evidently, $\psi^0$ given by \eq{KBLnuneq01} is analytic in $\bC\setminus i\bR$, and  $\forall\ \varphi\in (-\pi/2,\pi/2)$, \eq{asympsisRLPE} holds with
\bbe\label{ascofnupeqnumcc}
c_\infty(\varphi)=-2c\Ga(-\nu)\cos(\nu\pi/2)e^{i\nu\varphi}.
\ee
}
\end{example}
In \cite{EfficientAmenable}, we defined a class of Stieltjes-L\'evy processes (SL-processes). In order to save space, we do not reproduce the complete set of definitions. Essentially, $X$ is called a (signed) SL-process if $\psi$ is of the form
\bbe\label{eq:sSLrepr}
\psi(\xi)=(a^+_2\xi^2-ia^+_1\xi)ST(\cG^0_+)(-i\xi)+(a^-_2\xi^2+ia^-_1\xi)ST(\cG^0_-)(i\xi)+(\sg^2/2)\xi^2-i\mu\xi, 
\ee
where $ST(\cG)$ is the Stieltjes transform of a (signed) Stieltjes measure $\cG$,  $a^\pm_j\ge 0$, and $\sg^2\ge0$, $\mu\in\bR$.
We call a (signed) SL-process $X$ SL-regular if $X$ is SINH-regular. We proved in \cite{EfficientAmenable} that if $X$ is a (signed) SL-process then $\psi$ admits analytic continuation to the complex plane with two cuts along the imaginary axis, and
if $X$ is a SL-process, then, for any $q>0$, equation $q+\psi(\xi)=0$ has no solution on $\bC\setminus i\bR$.
We also proved that all popular classes of L\'evy processes bar the Merton model and Meixner processes are regular SL-processes, with $\ga_\pm=\pm \pi/2$;
the Merton model and Meixner processes are regular signed SL-processes, and $\ga_\pm=\pm \pi/4$.
 For lists of SINH-processes and SL-processes, with calculations of the order and type,
see \cite{EfficientAmenable}.

Note the following simple but important property of the characteristic exponent of a SINH-regular process, which is immediate from the
Cauchy integral formula.
\begin{lem}\label{lem:der_psi0}
Let $X$ be a SINH-regular L\'evy process  (on $\bR$) of   order
 $\nu$ and type $((\mum,\mup);\cC; \cC_+)$. Then for any closed cone $\cC'\subset (\bC\cup\{0\})$,
 segment $[\mu'_-,\mu'_+]\subset (\mum,\mup)$ and $n\in \bZ_+$, there exists $C_n(\mumpr,\muppr, \cC')$ such that
 \bbe\label{eq:der_psi0}
|(\psi^0)^{(n)}(\xi)|\le C_n(\mu'_-,\mu'_+, \cC')(1+|\xi|)^{\nu-n},\ \xi\in i[\mu'_-,\mu'_+]+\cC'.
\ee
\end{lem}

\subsection{Wiener-Hopf factorization and  the Laplace transform of $V_{f.t.d.}(h,t)$ and $V_2(a_1,h; T)$}\label{gen_form_barrier}
For $q>0$, let $T_q$ be an exponentially distributed random variable of mean $1/q$, independent of $X$.
For $\xi\in\bR$, set $\phipq(\xi)=\bE[e^{i\barX_{T_q}\xi}]$ and $\phimq(\xi)=\bE[e^{i\uX_{T_q}\xi}]$. The Wiener-Hopf factorization
formula is \bbe\label{whf0}
 \frac{q}{q+\psi(\xi)}=\phipq(\xi)\phimq(\xi),\ \xi\in\bR
 \ee
 (see, e.g. \cite[p.81]{RW})).
Define the (normalized) expected present value operators (EPV-operators) $\cE^\pm_q$ and $\cEq$ by
$\cEpq u(x)=\bE[u(x+\barX_{T_q})],$  $\cEmq u(x)=\bE[u(x+\uX_{T_q})],$ and $\cEq u(x)=\bE[u(x+X_{T_q})]$. Operators $\cE^\pm_q$ and $\cEq$ are pseudo-differential operators (pdo) with the symbols $\phi^\pm_q(\xi)$ and $q/(q+\psi(\xi))$, respectively. 
Evidently, the EPV operators are bounded operators in $L_\infty(\bR)$; if $\psi$ admits analytic continuation to a strip around the real axis, the EPV-operators act in spaces with exponential weights \cite{NG-MBS,barrier-RLPE}. 
 In 
\cite{NG-MBS,barrier-RLPE}, using the operator analog
$
\cEq=\cEpq\cEmq=\cEmq\cEpq
$
of \eq{whf0}, we derived general pricing formulas for barrier options and first touch digitals, for a wide class of 
regular L\'evy processes of exponential type (RLPE). 
The first step is the evaluation of the corresponding perpetual options, equivalently, calculation of the Laplace transform
of the option price w.r.t. time to maturity, for {\em positive values of the spectral parameter.} In \cite{single}, the formulas 
for the perpetual barrier options are generalized for arbitrary L\'evy process. We list the special cases 
of mirror reflections of Lemmas 2.1, 2.3 in \cite{single} sufficient for the purposes of the present paper. 

For $h\in \bR$, let $\tau^+_h$ denote the first entry time by $X$ into $[h,+\infty)$.
\begin{lem}\label{basic_formulas} 
Let $q>0$, $h>0$ and $G\in L_\infty$. Then 
\beqa\label{eq:ftd}
\bE[\tau^+_h\le T_q\ |\ X_0=0]&=&(\cEpq \bfo_{[h,+\infty)})(0),\\\label{eq:jointG}
\bE[G(X_{T_q})\bfo_{(-\infty,h]}(\barX_{T_q})\ |\ X_0=\barX_0=0]
&=&(\cEpq\bfo_{(-\infty,h]}\cEmq G)(0).
\eqa
In particular, for $a\le h$,
\bbe\label{eq:joint}
\bE[\bfo_{(-\infty,a]}(X_{T_q})\bfo_{(-\infty,h]}(\barX_{T_q})\ |\ X_0=\barX_0=0]
=(\cEpq\bfo_{(-\infty,h]}\cEmq\bfo_{(-\infty,a]})(0).
\ee
\end{lem}
It follows from \eq{eq:ftd} and \eq{eq:joint} that, for $q>0$, the Laplace transforms of $V_{f.t.d}(h,t)$ and
$V_2(a_1,h; T)$  w.r.t. $t$ and  $T$, respectively, are given by
\beqa\label{tVft}
\tV_{f.t.d}(h,q)&=&q^{-1}(\cEpq \bfo_{[h,+\infty)})(0),\\\label{tVjoint}
\tV_2(a_1,h;q)&=&-q^{-1}(\cEpq\bfo_{[h,+\infty)}\cEmq\bfo_{(-\infty,a_1]})(0).
\eqa
It is evident that each side of \eq{tVft} and \eq{tVjoint} admits analytic continuation w.r.t. $q$ to $\{\Re q>0\}$, therefore, one can recover $V_{f.t.d.}(h,t)$ and $V_2(a_1,h; T)$ using the Bromwich integral: 
\bbe\label{Brom}
 V(T)=\frac{1}{2\pi i}\int_{\Im q=\sg}\frac{e^{qT}}{q}\tV(q)dq,
 \ee
for $T>0$ and any $\sg>0$.
 Unless the function $V$ is  sufficiently regular, the representation \eq{Brom} is valid in the sense of the generalized functions only.
To study the regularity of $\tV_{f.t.d}(h,q)$ and $\tV_2(a_1,h;q)$ and design numerical procedures,
we expand the indicator functions in \eq{tVft} and \eq{tVjoint} 
 into the Fourier integrals, and use the following integral representations and properties of $\phi^\pm_q(\xi)$
  (see, e.g., \cite{NG-MBS,single,paraLaplace,paired,Contrarian,EfficientLevyExtremum}).

 \begin{lem}\label{phipmq_anal_cont}
 Let  there exist  $\mum\le 0\le \mup$, $\mum<\mup$, such that  $\bE[e^{-\ga X_1}]<\infty$, $\forall\ \ga\in [\mum,\mup]$. 
 
 Then
 \begin{enumerate}[(i)]
 \item
 $\psi(\xi)$ admits analytic continuation\footnote{Recall that a function $f$ is said to be analytic in the closure of an open set $U$ if $f$ is analytic in the interior of $U$ and continuous up to the boundary of $U$.} to the strip $S_{[\mum,\mup]}:=\{\xi\in \bC\ |\
 \Im\xi\in [\mum,\mup]\}$; 
 \item 
 let $\sg>\max\{-\psi(i\mum), -\psi(i\mup)\}$. Then there exists $c>0$  s.t. $|q+\psi(\xi)|\ge c$ for $ q\ge \sg$ and $\xi\in S_{[\mum,\mup]}$;
 \item
   let  $q\ge \sg$. Then  $\phipq(\xi)$ (resp., $\phimq(\xi)$) admits analytic continuation to  $\{\Im\xi\ge \mum\}$
   (resp., $\{\Im\xi\le \mup\}$) given by    \beqa\label{analcontphipq}
   \phipq(\xi)&=&\frac{q}{(q+\psi(\xi))\phimq(\xi)}, \ \Im\xi\in [\mum,0],
   \\\label{analcontphimq}
   \phimq(\xi)&=&\frac{q}{(q+\psi(\xi))\phipq(\xi)}, \ \Im\xi\in [0,\mup];
   \eqa
  \item   $\phipq(\xi)$ (resp., $\phimq(\xi)$) is uniformly bounded on  $\{\Im\xi\ge \mum\}$
   (resp., $\{\Im\xi\le \mup\}$).
 
   \end{enumerate}
    \end{lem}

 \begin{lem}\label{phipmq_explicit1}
 Let $\mu_\pm$, $X$ and $q$ satisfy the conditions of Lemma \ref{phipmq_anal_cont}. Then
 \begin{enumerate}[(a)]
 \item
 for any $\omm\in (\mum,\mup)$ and $\xi\in\{\Im\xi>\omm\}$,
 \bbe\label{phip1}
 \phipq(\xi)=\exp\left[\frac{1}{2\pi i}\int_{\Im\eta =\omm}\frac{\xi\ln(1+\psi(\eta)/q)}{\eta(\xi-\eta)}d\eta\right];
 \ee
  \item
 for any $\omp\in (\mum,\mup)$ and $\xi\in\{\Im\xi<\omp\}$,
 \bbe\label{phim1}
 \phimq(\xi)=\exp\left[-\frac{1}{2\pi i}\int_{\Im\eta=\omp}\frac{\xi\ln(1+\psi(\eta)/q)}{\eta(\xi-\eta)}d\eta\right].
 \ee
\end{enumerate}
\end{lem}

The integrands above decay very slowly at infinity, hence,  fast and accurate numerical realizations are impossible unless additional tricks
are used. If $X$ is SINH-regular, the rate of decay can be greatly increased 
 using appropriate conformal deformations of the line of integration
 and the corresponding changes of variables. Assuming that in Definition \ref{def:SINH_reg_proc_1D0},  $\ga_\pm$ are not extremely small in absolute value (and, in the case of  regular SL-processes, $\ga_\pm=\pm \pi/2$ are not small), the most efficient change of variables
is the sinh-acceleration 
   \bbe\label{eq:sinh}
 \eta=\chi_{\om_1  ,b,\om}(y)=i\om_1  +b\sinh(i\om+y), 
\ee
where $\om\in (-\pi/2,\pi/2)$, $\om_1  \in \bR, b>0$.
Typically, the sinh-acceleration is the best choice even if $|\ga_\pm|$ are of the order of $10^{-5}$.  The parameters $\om_1  ,b,\om$ are chosen so that the contour $\cL_{\om_1  ,b,\om}:=\chi_{\om_1  ,b,\om}(\bR)\subset i(\mup,\mup)+(\cC_{\gam,\gap}\cup\{0\})$ and, in the process of deformation, $\ln(1+\psi(\eta)/q)$ is a well-defined analytic function on a domain in $\bC$ or  the appropriate Riemann surface.  
We write $\cL^\pm_{\om_1  ,b,\om}$ or $\cL^\pm$ instead $\cL_{\om_1  ,b,\om}$ when we wish to indicate that $\pm \om>0$. Note that the wings of contours of type $\cL^+_{\om_1  ,b,\om}$ (resp., $\cL^-_{\om_1  ,b,\om}$) point upwards (respectively, downwards).
\begin{lem}\label{lem:WHF-SINH}
Let $X$ be SINH-regular of type $((\mum,\mup), \cC_{\gam,\gap}, \cC_{\gampr,\gappr})$. 

Then  there exists $\sg>0$ s.t. for all $q>\sg$, 
\begin{enumerate}[(i)]
\item
$\phipq(\xi)$ admits analytic continuation to $i(\mum,+\infty)+i(\cC_{\pi/2-\gam}\cup\{0\})$. For any $\xi\in i(\mum,+\infty)+i(\cC_{\pi/2-\gam}\cup\{0\})$, and any contour 
$\cL^-_{\om_1  ,b,\om}\subset i(\mum,\mup)+(\cC_{\gam,\gap}\cup\{0\})$ lying below $\xi$,
\bbe\label{phipq_def}
\phipq(\xi)=\exp\left[\frac{1}{2\pi i}\int_{\cL^-_{\om_1  ,b,\om}}\frac{\xi \ln (1+\psi(\eta)/q)}{\eta(\xi-\eta)}d\eta\right];
\ee
\item
$\phimq(\xi)$ admits analytic continuation to $i(-\infty,\mup)-i(\cC_{\pi/2+\gap}\cup\{0\})$. For any $\xi\in i(\-\infty,\mup)-i(\cC_{\pi/2+\gap}\cup\{0\})$, and any contour 
$\cL^+_{\om_1  ,\om,b}\subset i(\mum,\mup)+(\cC_{\gam,\gap}\cup\{0\})$ lying above $\xi$,
\bbe\label{phimq_def}
\phimq(\xi)=\exp\left[-\frac{1}{2\pi i}\int_{\cL^+_{\om_1  ,\om,b}}\frac{\xi \ln (1+\psi(\eta)/q)}{\eta(\xi-\eta)}d\eta\right].
\ee
\end{enumerate}
\end{lem}
The integrals are efficiently evaluated making the change of variables $\eta=\chi_{\om_1,b,\om}(y)$ and applying the simplified trapezoid rule.

\begin{rem}\label{rem:SL-WHF}{\rm 
In the process of deformation, the expression $1+\psi(\eta)/q$ may not assume value zero. In order to avoid complications stemming from analytic continuation to an appropriate Riemann surface, it is advisable to ensure that $1+\psi(\eta)/q\not\in(-\infty,0]$.
 Thus, if $q>0$ - and only positive $q$'s are used in the Gaver-Stehfest method or GWR algorithm - and $X$ is an SL-process, 
any $\om\in (0,\pi/2)$ is admissible in \eq{phipq_def}, and any $\om\in (-\pi/2,0)$ is admissible in \eq{phimq_def}. 
If the sinh-acceleration is applied to the Bromwich integral, then additional conditions on $\om$ must be imposed. See 
Sect. \ref{ss:anal_cont_q}.
}
\end{rem}
\subsection{Decomposition of the Wiener-Hopf factors}\label{WHF_decomp}
In the remaining part of the paper, we assume that the Wiener-Hopf factors $\phi^\pm_q(\xi), q>0,$ admit
the representations $\phi^\pm_q(\xi)=a^\pm_q+\phi^{\pm\pm}_q(\xi)$ and $\cE^\pm_q=a^\pm_qI+\cE^{\pm\pm}_q$,
where $a^\pm_q\ge 0$, and $\phi^{\pm\pm}_q(\xi)$ satisfy the bounds
\beqa\label{WHFdecayP}
 |\phi^{++}_q(\xi)|&\le & C_+(q)(1+|\xi|)^{-\nup},\ \Im \xi\ge \mum,\\
 \label{WHFdecayM}
 |\phi^{--}_q(\xi)|&\le & C_+(q)(1+|\xi|)^{-\num},\ \Im \xi\le \mup,
 \eqa
 where $\nu_\pm>0$ and $C_\pm(q)>0$ are independent of $\xi$. These conditions are satisfied for all popular classes of L\'evy processes
 bar the driftless Variance Gamma model.

The following more detailed properties of the Wiener-Hopf factors are established in \cite{NG-MBS,BLSIAM02,barrier-RLPE}
for the class of RLPE (Regular L\'evy processes of exponential type); the proof for SINH-regular processes is the same only $\xi$ is allowed to tend to $\infty$ not only in the strip of analyticity but in the union of a strip and cone. See \cite{BIL,asymp-sens,paired} for the proof of the statements below for
several classes of SINH-regular processes (the definition of the SINH-regular processes formalizing properties used in \cite{BIL,asymp-sens,paired} was suggested in \cite{SINHregular} later). The contours  in Lemma \ref{lem:atoms} below
are in a domain of analyticity s.t. $q-i\mu\xi\neq 0$ and $1+\psi^0(\xi)/(q-i\mu\xi)\not\in (-\infty,0]$. These restrictions on the contours are needed when $\psi^0(\xi)=O(|\xi|^\nu)$ as $\xi\to\infty$ in the domain of analyticity and $\nu<1$. Clearly, in this case, for sufficiently large $q>0$, the condition holds.
If $X$ is an RLPE but not SINH-regular, the contours of integration in the lemma below are straight lines in the strip of analyticity. 

\begin{lem}\label{lem:atoms}
Let $\mum<0<\mup$,  $q>0$, let $X$ be SINH-regular  of  type 
$((\mum,\mup), \cC_{\gam,\gap}, \cC_{\gampr,\gappr})$, and order $\nu$.  Then
\begin{enumerate}[(1)]
\item
if  $\nu\in [1,2]$ or $\nu\in (0,1)$ and the drift  $\mu=0$, then neither
 $\barX_{T_q}$ nor $\uX_{T_q}$ has an atom at 0, and $\phi^\pm_q(\xi)$ admit the bounds \eq{WHFdecayP} and 
 \eq{WHFdecayM}, 
 where $\nu_\pm>0$ and $C_\pm(q)>0$ are independent of $\xi$;
 \item
  if  $\nu\in [0+,1)$  and $\mu>0$, then
 \begin{enumerate}[(a)]
 \item
 $\barX_{T_q}$ has no atom at 0 and $\uX_{T_q}$ has an atom  $a^-_q\de_0$ at zero, where
 \bbe\label{eq:cmqp}
a^-_q=\exp\left[-\frac{1}{2\pi i}\int_{\cL^+_{\om_1  ,b,\om}}\frac{\ln((1+\psi^0(\eta)/(q-i\mu\eta))}{\eta}d\eta\right],
\ee
and $\cL^+_{\om_1  ,b,\om}$ is a contour as in Lemma \ref{lem:WHF-SINH} (ii); 
 \item
for $\xi$ and $\cL^-_{\om_1  ,b,\om}$ in Lemma \ref{lem:WHF-SINH} (i), $\phipq(\xi)$ admits the representation
 \bbe\label{phip1finvar}
 \phipq(\xi)=\frac{q}{q-i\mu\xi}\exp\left[\frac{1}{2\pi i}\int_{\cL^-_{\om_1  ,b,\om}}\frac{\xi\ln(1+\psi^0(\eta)/(q-i\mu\eta))}{\eta(\xi-\eta)}d\eta\right],
 \ee
and  satisfies the bound \eq{WHFdecayP} with $\nup=1$;
 \item
 $\phimq(\xi)=a^-_q+\phi^{--}_q(\xi)$, where $\phi^{--}_q(\xi)$ satisfies 
 \eq{WHFdecayM} with arbitrary $\num\in (0,1-\nu)$;
 \item
 $\cEmq=a^-_qI+\cE^{--}_q$, where $\cE^{--}_q$ is the pdo with the symbol $\phi^{--}_q(\xi)$;
 \end{enumerate}
  \item
  if  $\nu\in [0+,1)$  and $\mu<0$, then
 \begin{enumerate}[(a)]
 \item
 $\uX_{T_q}$ has no atom at 0 and $\barX_{T_q}$ has an atom $a^+_q\de_0$ at zero, where
 \bbe\label{eq:cpqp}
a^+_q=\exp\left[\frac{1}{2\pi i}\int_{\cL^-_{\om_1  ,b,\om}}\frac{\ln((1+\psi^0(\eta)/(q-i\mu\eta))}{\eta}d\eta\right],
\ee
and $\cL^-_{\om_1  ,b,\om}$ is a contour as in Lemma \ref{lem:WHF-SINH} (i);
 \item
for $\xi$ and $\cL^+_{\om_1  ,b,\om}$ in Lemma \ref{lem:WHF-SINH} (ii), $\phimq(\xi)$ admits the representation
 \bbe\label{phim1finvar}
 \phimq(\xi)=\frac{q}{q-i\mu\xi}\exp\left[-\frac{1}{2\pi i}\int_{\cL^+_{\om_1  ,b,\om}}\frac{\xi\ln(1+\psi^0(\eta)/(q-i\mu\eta))}{\eta(\xi-\eta)}d\eta\right],
 \ee
and
 satisfies the bound \eq{WHFdecayM} with $\num=1$;
 \item
 $\phipq(\xi)=a^+_q+\phi^{++}_q(\xi)$, where $\phi^{++}_q(\xi)$ satisfies  
 \eq{WHFdecayP} with arbitrary $\nup\in (0,1-\nu)$;
\item
$\cEpq=a^+_qI+\cE^{++}_q$, where $\cE^{++}_q$ is the pdo with the symbol $\phi^{++}_q(\xi)$.
 \end{enumerate}
 \end{enumerate}
\end{lem}

\subsection{Analytic continuation w.r.t. $q$}\label{ss:anal_cont_q} It is easy to see that,
given a strip of analyticity $S_{[\mu'_-,\mu'_+]}$, $\mu'_-<0<\mu'_+$, of $\psi$, one can find  $\sg_0>0$ such that 
 both $\phipq(\xi)$ and $\phimq(\xi)$ are uniformly bounded on $\cU(\sg; \mu'_-,\mu'_+):=\{\Re q\ge \sg_0, \Im\xi\in [\mu'_-,\mu'_+]\}$,
 $\Re(q+\psi(\xi))$ is positive on $\cU(\sg;  \mu'_-,\mu'_+)$, and the functions in the formulas for $\phi^\pm_q(\xi)$ above, with flat contours of integration,  admit analytic 
 continuation to $\cU(\sg; \mu'_-,\mu'_+)$. Therefore, any formula for the Laplace transforms $\tV(q)$ of  expectations of random variables which we use or derive below in terms of the Wiener-Hopf factors can be used to define analytic continuation of 
 $\tV(q)$ to the half-plane $\{\Re q\ge \sg_0\}$ and use the Bromwich integral
 to recover $V$: for any $T>0$ and $\sg\ge \sg_0$, \eq{Brom} holds.
 In order that a deformation of the contour in \eq{Brom} can be justified, and a numerical procedure for the evaluation of the Bromwich integral \eq{Brom} be reliable, it is necessary that
$\tV(q)$ decay sufficiently fast as $q\to\infty$ along the contour of integration. For the proof, we use the integration by parts in the formula
for $\tV(q)$ and the following lemma. First, we consider the case $\nu\in [1,2]$ or $\nu\in (0,1)$ and $\mu=0$. If $\nu\neq 1$,
set $\nu_\pm=\nu/2$. If $\nu=1$, we assume that $c_{\infty}(\varphi)$ in \eq{asympsisRLPE} is of the form
$c_\infty e^{i\varphi}$, then $(\nup,\num)$ is the solution of the system $\nup+\num=1, \nup-\num=(2/\pi)\arctan(\mu/c_\infty)$.

\begin{lem}\label{lem:der_WHF_q_SINH}
Let $X$ be a SINH-regular process. Let either the order $\nu\in [1,2]$ or  $\nu\in (0,1)$ and the drift is 0. Then
there exist $(\mu'_-,\mu'_+)\subset(\mum,\mup)$, $\mu'_-<\mu'_+$, a cone $\cC_{\ga'_-,\ga'_+}$, $\ga'_-<0<\ga'_+$, and $\sg_0>0, \om_L\in (0,\pi/2)$ such that
\begin{enumerate}[(a)]
\item
for all $q\in \sg_0+\cC_{\pi/2+\om_L}$ and $\xi\in i(\mu'_-,\mu'_+)+(\cC_{\ga'_-,\ga'_+}\cup \{0\})$,
\bbe\label{good_bound_psiq}
q+\psi(\xi)\not\in  (-\infty,0];
\ee
\item
$\phipq(\xi)$ admits analytic continuation to $(\sg_0+\cC_{\pi/2+\om_L})\times (i(\mu,+\infty)+i(\cC_{\pi/2-\ga'_-}\cup\{0\}))$
and obeys the bounds
\beqa\label{good_bound_phipq}
|\phipq(\xi)|&\le& C_+(|q|^{1/\nu}+|\xi|)^{-\nup},\\\label{good_bound_phipq_der}
|\dd_q^m\dd^n_\xi\phipq(\xi)|&\le& C_{+,m,n}(|q|^{1/\nu}+|\xi|)^{-\nup}|q|^{-m}(1+|\xi|)^{-n},\ n,m\in \bZ_+,
\eqa
where $C_+, C_{+,m,n}$ are independent of $q,\xi$;
\item
$\phimq(\xi)$ admits analytic continuation to $(\sg_0+\cC_{\pi/2+\om_L})\times (i(-\infty,\mup)-i(\cC_{\pi/2+\ga'_+}\cup\{0\}))$
and obeys the bounds
\beqa\label{good_bound_phimq}
|\phipq(\xi)|&\le &C_-(|q|^{1/\nu}+|\xi|)^{-\num},\\\label{good_bound_phimq_der}
|\dd_q^m\dd^n_\xi\phipq(\xi)|&\le & C_{-,m,n}(|q|^{1/\nu}+|\xi|)^{-\num}|q|^{-m}(1+|\xi|)^{-n},
\eqa
where $C_+, C_{+,m,n}$ are independent of $q,\xi$.
\end{enumerate}

\end{lem}
 \begin{proof}
 Eq. \eq{good_bound_psiq}, \eq{good_bound_phipq} and \eq{good_bound_phimq} were derived and used in \cite{NG-MBS,barrier-RLPE,BIL} for $\xi$ in a strip, and in \cite{paired} for
 $\xi$ in a union of a strip and cone. The bounds  \eq{good_bound_phipq} and \eq{good_bound_phimq} (in a less explicit form)
 are proved and used in \cite{Contrarian,EfficientLevyExtremum}.
 The bounds \eq{good_bound_phipq_der} and \eq{good_bound_phimq_der} are 
 immediate from \eq{good_bound_phipq}, \eq{good_bound_phimq} and the
Cauchy integral formula.

 \end{proof}

\section{Evaluation and regularity of  $V_{f.t.d.}(h,t)$, $V_2(a_1,h; T)$, and  $\dd_hV_2(a_1,h; T)$}\label{s: basic ingredients and regularity}

\subsection{Representations and numerical evaluation of $V_{f.t.d.}(h,t)$}\label{ss:Vftd}
Let $h>0$ and let $\psi$ admit analytic continuation to $S_{(\mum,0]}$. Then, for any $\omm\in (\mum,0)$ and sufficiently large
$\sg$, 
\beqa\label{V_ftd}
V_{f.t.d.}(h,t)&=&\frac{1}{2\pi i}\int_{\Re q=\sg}dq\frac{e^{qt}}{q}\frac{1}{2\pi}\int_{\Im\xi=\omm}d\xi\,e^{-i\xi h}\frac{\phipq(\xi)}{i\xi}
\\\label{V_ftd2}
&=&\frac{1}{2\pi i}\int_{\Re q=\sg}dq\frac{e^{qt}}{q}\frac{1}{2\pi}\int_{\Im\xi=\omm}d\xi\,e^{-i\xi h}\frac{\phi^{++}_q(\xi)}{i\xi}.
\eqa
If $\phipq(\xi)\neq \phi^{++}_q(\xi)$, it is advantageous to use \eq{V_ftd2} rather than \eq{V_ftd} because on the strength of \eq{WHFdecayP},
the integral on the RHS of \eq{V_ftd2} is absolutely convergent whereas the one on the RHS of \eq{V_ftd} is not.
 If $X$ satisfies the conditions of Lemma \ref{lem:der_WHF_q_SINH}, we can deform the inner contour into a contour of the form $\cL^-_{\om_1  ,b,\om}$:
  \bbe\label{V_ftd_sinh}
V_{f.t.d.}(h,t)=\frac{1}{2\pi i}\int_{\Re q=\sg}\frac{e^{qt}}{q}\frac{1}{2\pi}\int_{\cL^-_{\om_1  ,b,\om}}e^{-i\xi h}\frac{\phi^{++}_q(\xi)}{i\xi}d\xi,
\ee
make the corresponding sinh-change of variables,
 and apply the simplified trapezoid rule. For each $q$ used in a numerical method for the evaluation of the Bromwich integral, 
the error tolerance of the order of E-12-E-13 can be satisfied using the simplified trapezoid rule with
 150-300 terms (the number depends on the properties of $\psi$, the opening angle of the sector of analyticity especially).

 To calculate the outer integral, we apply the sinh-acceleration or  
 summation by parts in the infinite trapezoid rule and truncate the sum. The error tolerance of
 the order of E-12 (resp., E-14) can be satisfied using a truncated sum with 150-200 (resp., 200-250) terms. We can also apply the GWR algorithm with $2M=16$ terms but then the best accuracy that can be achieved is of the order of E-07 unless high precision arithmetics and $2M>16$ are used.
 
Thus,  if the order of the process $\nu \in [1,2]$ or $\nu\in (0,1)$ and $\mu=0$, we recommend to apply the sinh-acceleration to the outer integral as well:
  \bbe\label{V_ftd_sinh_sinh}
V_{f.t.d.}(h,t)=\frac{1}{2\pi i}\int_{\cL^L_{\sg,b_\ell,\om_\ell}}dq\frac{e^{qt}}{q}\frac{1}{2\pi}
\int_{\cL^-_{\om_1  ,b,\om}}d\xi\,e^{-i\xi h}\frac{\phi^{++}_q(\xi)}{i\xi},
\ee
 where $\cL^L_{\sg,b_\ell,\om_\ell}=\chi^L_{\sg,b_\ell,\om_\ell}(\bR)$ and $\chi^L_{\sg,b_\ell,\om_\ell}(y)=\sg+ib_\ell\sinh(i\om_\ell+y)$.
 The parameters are chosen so that, for all $(q,\xi)$ arising in the process of deformations, $q+\psi(\xi)\not\in (-\infty,0]$. 
If $\nu\in (1,2]$ or $\nu\in (0,1]$ and $\mu=0$, the crucial parameters $\gam<\om<0$ and $0<\om_\ell<\pi/2$  must satisfy $\max\{1,\nu\}|\om|<\pi/2-\om_\ell$ (if $\nu=1$ and $\mu\neq 0$, the condition is more involved). 
If  $\max\{1,\nu\}|\om|<\pi/2-\om_\ell$,
it is straightforward to show that there exist $\om_1  ,b, \sg, b_\ell$ such that for all $(q,\xi)$ arising in the process of deformations, $q+\psi(\xi)\not\in (-\infty,0]$.
 For details, see \cite{Contrarian,EfficientLevyExtremum}.

 \subsection{Representations and numerical evaluation of $V_2(a_1,h; T)$}\label{ss:V2}
 The following representation for $V_2(a_1,h; T)$ is immediate from \cite[Thm.3.5]{EfficientLevyExtremum}:
      \begin{thm}\label{thm:CPDF}  Let $a_1\le b$, $h>0$. For any $\mum<\omm<0<\omp  <\mup$, 
      there exists $\sg_0>0$ such that, for any $q\ge \sg_0$, the Laplace transform $\tV_2(a_1,h;q)$ is given by
      \bbe\label{tV2}
      \tV_2(a_1,h; T)=\frac{1}{(2\pi)^2 q}\int_{\Im\eta=\omm}d\eta\, e^{-ih\eta}\phi^{++}_q(\eta)  \int_{\Im\xi=\omp}d\xi\, \frac{e^{i\xi(h-a_1)}\phi^{--}_q(\xi)}{\xi(\xi-\eta)}.
\ee 
    \end{thm}
In \cite{EfficientLevyExtremum}, it is also proved that the integrand on the RHS of \eq{tV2} is bounded (in absolute value) by an integrable function independent
of $q\in \{\Re q\ge \sg_0\}$. Therefore, $ V_2(a_1,h; T)$ can be recovered using the Bromwich integral
\beqa\label{V2}
      V_2(a_1,h; T)&=&\frac{1}{2\pi i}\int_{\Re q=\sg}dq\,\frac{e^{qT}}{q}\\\nonumber
&&\cdot\frac{1}{(2\pi)^2}\int_{\Im\eta=\omm}d\eta\,e^{-ih\eta}\phi^{++}_q(\eta)  \int_{\Im\xi=\omp  }d\xi\, \frac{e^{i\xi(h-a_1)}\phi^{--}_q(\xi)}{\xi(\xi-\eta)}
\eqa 
understood in the sense of generalized functions. 
Assuming that $\nu\in [1,2]$ or $\nu\in (0,1)$ and $\mu=0$, we can  apply appropriate sinh-acceleration deformations to the integrals w.r.t.
$q,\eta,\xi$: 
\beqa\label{V2_sinh}
      V_2(a_1,h; T)&=&\frac{1}{2\pi i}\int_{\cL^L_{\sg,b_\ell,\om_\ell}}dq\,\frac{e^{qT}}{q}\\\nonumber
&&\cdot\frac{1}{(2\pi)^2}\int_{\cL^-_{\om^-_1,b^-,\om_-}}d\eta\,e^{-ih\eta}\phi^{++}_q(\eta)  \int_{\cL^+_{\om^+_1,b^+,\om_+}} \frac{e^{i\xi(h-a_1)}\phi^{--}_q(\xi)d\xi\,}{\xi(\xi-\eta)}.
\eqa 
For any compact subset of $\{a_1\le h, h\ge 0, T\ge t>0\}$, the triple integrand is bounded in absolute value by a function of class $L_1$ (see \cite{EfficientLevyExtremum} for the proof),  hence, the function on the RHS  is continuous on $\{a_1\le h, h\ge 0, T>0\}$.
We can also apply the GWR algorithm to the outer integral below
\beqa\label{V2_summ}
      V_2(a_1,h; T)&=&\frac{1}{2\pi i}\int_{\Re q=\sg}dq\,\frac{e^{qT}}{q}\\\nonumber
&&\cdot 
      \frac{1}{(2\pi)^2}\int_{\cL^-_{\om^-_1,b^-,\om_-}}d\eta\,e^{-ih\eta}\phi^{++}_q(\eta) \int_{\cL^+_{\om^+_1,b^+,\om_+}} \frac{e^{i\xi(h-a_1)}\phi^{--}_q(\xi)d\xi\,}{\xi(\xi-\eta)},
\eqa 
Numerical examples in \cite{EfficientLevyExtremum}
demonstrate that the best accuracy achievable with the GWR algorithm is several orders of magnitude worse, and, for the same accuracy, the CPU times are approximately the same.

\subsection{Representations of $\dd_hV_2(a_1,h; T)$ and $\dd_{a_1}\dd_hV_2(a_1,h; T)$}\label{ss:V2der}
If $\nu\in [1,2]$ or $\nu\in (0,1)$ and $\mu=0$, the differentiation under the integral sign on the RHS of \eq{V2_sinh}
is justified\footnote{The proof in \cite{asymp-sens} for several classes of SINH-regular processes  is valid for all SINH-regular processes.}, and we obtain
\beqa\label{derV2_sinh}
      \dd_hV_2(a_1,h; T)&=&\frac{1}{2\pi i}\int_{\cL^L_{\sg,b_\ell,\om_\ell}}dq\,\frac{e^{qT}}{q}\\\nonumber
&&\cdot\frac{1}{(2\pi)^2}\int_{\cL^-_{\om^-_1,b^-,\om_-}}d\eta\,e^{-ih\eta}\phi^{++}_q(\eta)  \int_{\cL^+_{\om^+_1,b^+,\om_+}} \frac{e^{i\xi(h-a_1)}\phi^{--}_q(\xi)d\xi\,}{-i\xi},
\eqa 
and 
\beqa\label{derV2_sinh_dera}
     \dd_{a_1} \dd_hV_2(a_1,h; T)&=&\frac{1}{2\pi i}\int_{\cL^L_{\sg,b_\ell,\om_\ell}}dq\,\frac{e^{qT}}{q}
     \frac{1}{(2\pi)^2}\int_{\cL^-_{\om^-_1,b^-,\om_-}}d\eta\,\\\nonumber
&&\cdot e^{-ih\eta}\phi^{++}_q(\eta)  \int_{\cL^+_{\om^+_1,b^+,\om_+}} e^{i\xi(h-a_1)}\phi^{--}_q(\xi)d\xi.
\eqa 
In the general case, we can derive an analogue of \eq{derV2_sinh} in the sense of the generalized functions only:
\beqa\label{derV2}
     \dd_h V_2(a_1,h; T)&=&\frac{1}{2\pi i}\int_{\Re q=\sg}dq\,\frac{e^{qT}}{q}\\\nonumber
&&\cdot\frac{1}{(2\pi)^2}\int_{\Im\eta=\omm}d\eta\,e^{-ih\eta}\phi^{++}_q(\eta)  \int_{\Im\xi=\omp  } \frac{e^{i\xi(h-a_1)}\phi^{--}_q(\xi)d\xi\,}{-i\xi}.
\eqa 

\begin{rem}\label{rem:nu<1}{\rm 
 As functions of $\xi$ on the contours that we consider,  $\phi^{\pm\pm}_q(\xi)/\xi$ are absolutely integrable. In the result, in the formulas for $V_{f.t.d.}(h,t)$ and $V_2(a_1,h; T)$, the integrand is absolutely integrable as a function of $\xi$. However, to make the integrand absolutely convergent as a function of $q$, we need to integrate by parts. Under conditions of Lemma \ref{lem:der_WHF_q_SINH}, we can integrate by parts, 
 and reduce the analysis to the case of the absolutely convergent multiple integrals. If $\nu<1$ and $\mu\neq 0$, then not only 
 Lemma \ref{lem:der_WHF_q_SINH} is not applicable; one can show that the conclusions of the lemma fail.
 Hence, the formulas for $V_{f.t.d.}(h,t)$ and $V_2(a_1,h; T)$, and, especially, for $\dd_hV_2(a_1,h; T)$, can be understood in
 the sense of generalized functions only, and  the proof of convergence and error bounds for
standard quadratures are difficult.
}
 \end{rem}
\subsection{Evaluation using Carr's randomization}\label{ss:Carr}
This problem does not arise if Carr's randomization
is applied because, at each step, prices of perpetual barrier options are calculated, and, for positive values of the spectral parameter, the Wiener-Hopf factors are sufficiently regular. 
 Carr's randomization algorithm
(maturity randomization) is justified for  Markov processes in \cite{MSdouble}, under a weak regularity condition,
and applied to price barrier options in regime-switching hyper-exponential jump-diffusion models.
In \cite{single}, explicit Carr's randomization algorithms were developed for RLPE  processes. 
The class of SINH-processes being a subclass of the class of RLPE processes, we can evaluate
$V_{f.t.d.}(h,t)$ and $V(a_1,h; T)$ using the method in \cite{single}. In \cite{asymp-sens}, the method is justified for 
sensitivities, hence,
one can evaluate $\dd_hV(a_1,h; T)$  as well. However, the calculations in \cite{single} are
in the state space, therefore, slower than the calculations in the dual space. Also, additional interpolation errors appear.

In Carr's randomization approximation, the maturity period $T$ of the option is divided into $N$
subintervals, using points $0=t_0<t_1<\dotsb<t_N=T$ (when $V_{f.t.d.}(h,t)$ is calculated, $T=t$), and 
each sub-period $[t_s,t_{s+1}]$ is replaced with an exponentially distributed
random maturity period with mean $\De_s=t_{s+1}-t_s$. Moreover,
these $N$ random maturity sub-periods are assumed to be independent
of each other and of the process $X$. In \cite{carr-random}, it is
assumed that $\De_s=T/N$ for all $s$, but, in principle, it is unnecessary 
to impose this requirement. In \cite{single}, the killing rate (interest rate) $r$ is assumed positive but
for options of finite maturity the proof and algorithms remain valid for any $r\in \bR$. For the purposes of the present paper,
we need $r=0$.   Below, $V_s$ denotes the Carr's randomization approximation to the value at time $t_s$.
\vskip0.1cm
\noindent
{\sc Algorithm for $V_{f.t.d.}(h,t;x)$, the price of the first touch digital at $X_0=x$. The maturity date $t$, the upper barrier $h$.} \begin{enumerate}[I.]
\item
Set $V_N=\bfo_{[h,+\infty)}$.
\item For $s=N-1,N-2,\ldots, 0$, calculate $q_s=t/\De_s$ and
$V_s=(1/q_s)\bfo_{(\infty,h)}\cE^+_{q_s}V_{s+1}+\bfo_{[h,+\infty)}$. 
\item
$V_{f.t.d.}(h,t;x)=V_0(x), x\in\bR.$
\end{enumerate}
\vskip0.1cm
\noindent
{\sc Algorithm for $V_{joint}(a_1,h;T,x)=\bP[X_T\le a_1, \barX_T\le h\ |\ X_0=x]$, $a_1\le h$.}
\begin{enumerate}[I.]
\item
Set $V_N=\bfo_{(-\infty,a_1]}$.
\item For $s=N-1,N-2,\ldots, 0$, calculate $q_s=T/\De_s$ and
$V_s=(1/q_s)\cE^+_{q_s}\bfo_{(-\infty,h]\cE^-_{q_s}}V_{s+1}$. 
\item
$V_{joint}(a_1,h;T,x)=V_0(x), x\in \bR.$
\end{enumerate}
The values $V_s(x)$ are evaluated at points of a chosen grid using an appropriate piece-wise interpolation 
procedure. The state space is truncated, and the approximation to $V_s$ is represented by the array
$\{V_{s;j}\}_{j=1}^M$ of the values $V_s(x_j)$ at points of the chosen grid. The EPV-operators become
the matrix operators, which can be efficiently realized using the fast convolution. See \cite{single} for details.
 To be more specific,   the matrix elements of the discretized $\cE^\pm_q$ are linear combinations
of reals of the form
\beqa\label{mat_elem_1}
(\cE^\pm_q \bfo_{(-\infty,a]}(\cdot)^n)(0)&=&\frac{(n-1)!}{2\pi}\int_{\Im\xi=\om_-} e^{-ia\xi}\phi^\pm_q(\xi)(i\xi)^{-n-1}d\xi,\\\label{mat_elem_2}
(\cE^\pm_q \bfo_{[a,+\infty)}(\cdot)^n)(0)&=&-\frac{(n-1)!}{2\pi}\int_{\Im\xi=\om_+} e^{-ia\xi}\phi^\pm_q(\xi)(i\xi)^{-n-1}d\xi,
\eqa
for $n\ge 0$, where $\omm<0$ and $0<\omp$ are sufficiently small in absolute value so that the contours $\{\Im\xi=\om_\pm\}$ are
in the strip of analyticity of $\phi^\pm_q$. In \cite{single}, uniforms grids and the piece-wise linear interpolation is used, and the matrix elements of the discretized EPV-operators are calculated using the Fast Fourier transform (FFT).
More accurate and faster calculations of the matrix elements can
be performed using the sinh-acceleration similarly to \cite{MarcoDiscBarr} where the matrix elements of the discretized
transition operators are calculated using fractional-parabolic deformations. 
If the process has a significant Brownian motion component or the characteristic exponent is a rational function, then
the interpolation of higher order can be advantageous. See \cite{MarcoDiscBarr}  for explicit formulas for the matrix elements.

If the value function is irregular, then, in a vicinity of the barrier, it is advantageous to approximate $V_s$ on
$[h-2\De, h]$ as $V_s(x)\approx V_s(h-\De)$ instead of $V_s(x)\approx V_s(h-)+(x-h)(V_s(h-0),V_s(h-2\De))/2$.
We leave to the reader the straightforward modification of the construction of the matrix approximation in \cite{single}.

\begin{rem}\label{rem:Carr}{\rm \begin{enumerate}[(a)]
\item If the left tail of the distribution of $X$ decays very slowly, then
it is necessary to use a very large truncated interval, and Carr's randomization becomes very inefficient.

\item One can perform all steps of Carr's randomization in the dual space using the fast Hilbert transform
as in \cite{feng-linetsky08} were barrier options in the discrete time model are priced. However,
$\phi^\pm_q(\xi)$ decay very slowly as $\xi\to\infty$, hence, extremely long grids are necessary to
satisfy even a moderate error tolerance.

\end{enumerate}
}
\end{rem}

\subsection{Other methods for pricing barrier options}\label{ss:other barrier}
If $\nu\in [1,2]$ or $\nu\in (0,1)$ and $\mu=0$, then the method of \cite{EfficientLevyExtremum} based on
the application of the sinh-acceleration to the Bromwich integral and the integrals w.r.t.  other variables in the pricing formula
is the most efficient. Several popular methods are based on either approximation of small jumps by an additional Brownian Motion component or by processes with rational characteristic exponents. The method of \cite{EfficientLevyExtremum}
is applicable after such an approximation\footnote{An approximation of a small jump component should be done so that the new process is SINH-regular.} but the approximation error itself is, typically, sizable, because, implicitly,
 functions with the unbounded first derivative are approximated by smooth functions. See \cite{KudrLev09,single} for examples of errors 
of Cont-Volchkova method (approximation of a small jump component) and approximation with the Hyper-Exponential jump-diffusion model, respectively.
One can also approximate barrier options with continuous monitoring with barrier options with discrete monitoring,
and use any method for pricing the latter, e.g., COS or BPROJ method, but both methods introduce additional  errors. See \cite{MarcoDiscBarr,BSINH} for examples and qualitative analysis.

\subsection{Regularity of $V_{f.t.d.}(h,t)$, $V_2(a_1,h; T)$ and
$\dd_hV_2(a_1,h; T)$}\label{ss:regularity} In \cite{NG-MBS} and \cite{BIL}, the following asymptotic formulas were derived
for $V_{f.t.d.}(h,t)$ and $V_2(a_1,h; T)$, respectively (in both papers, the spot varies and barrier are fixed;
we reformulate the results for the case when the spot 0 is fixed, and the barrier varies). In \cite{asymp-sens}, the asymptotic formulas are derived for sensitivities.

\begin{thm}\label{thm:as}
Let $X$ be SINH-regular of order $\nu$, and $t>0, T>0$, $a_1\le 0$ are fixed. 

Then 
there exist $\ka(t)>0$, $\rho(t)\in (0,1)$  and $\rho(a_1,T)>0$ (depending on the parameters of the process) s.t. as $h\downarrow 0$, 
\begin{enumerate}[(a)]
\item
if $\nu\in (1,2]$ or $\nu\in (0,1]$ and $\mu=0$,  
\beqa\label{eq_ft_as}
V_{f.t.d.}(h,t)&\sim& 1-\ka(t)h^{\nu/2},\\\label{eq_ft_as_der}
\dd_hV_{f.t.d.}(h,t)&\sim& (\nu/2)\ka(t)h^{\nu/2-1},\\\label{as_joint}
V_2(a_1,h; T)+\bP[X_T\le a_1]&\sim& \rho(a_1,T)h^{\nu/2}, \\\label{as_joint_der}
\dd_{h}V_2(a_1,h; T)&\sim& -(\nu/2)\rho(a_1,T)h^{\nu/2-1};
\eqa
\item
if $\nu\in [0+,1)$ and $\mu>0$,   
\beqa\label{eq_ft_as_m}
V_{f.t.d.}(h,t)&\sim& 1-\ka(t)h,\\\label{as_joint_m}
V_2(a_1,h; T)+\bP[X_T\le a_1]&\sim& \rho(a_1,T)h;
\eqa
\item
if $\nu\in [0+,1)$ and $\mu<0$,   
\beqa\label{eq_ft_as_p}
V_{f.t.d.}(h,t)&\sim& \rho(t),\\\label{as_joint_p}
V_2(a_1,h; T)+\bP[X_T\le a_1]&\sim& \rho(a_1,T).
\eqa
\item
if $\nu=1$, $\mu\neq 0$ and, in \eq{asympsisRLPE}, $c_\infty(\varphi)=c_{\infty,0} e^{i\varphi}$,
\footnote{This is the case for a popular class of Normal Inverse Gaussian processes.}
then \eq{eq_ft_as}-\eq{as_joint_der}
 are valid  with $\nup=1/2+\arctan(\mu/c_\infty)/\pi$ instead of $\nu/2=1/2$.
\end{enumerate}
\end{thm}
 If $\nu\in (0,1)$ and $\mu\neq 0$, then the asymptotics of $\dd_{h}V_2(a_1,h; T)$
is, in general, very irregular and depends on more detailed properties of the characteristic exponent
than the ones used in the definition of the class of SINH-regular processes.


\section{SINH-method}\label{s:main}
\subsection{Conditions for the   sinh-deformations}\label{ss: sinhU2}
 In the process of the contour deformations in this section, the following conditions must be satisfied for all  $q,\xi,\eta, q',\eta'$
 that appear in the formulas below:
\begin{enumerate}[(i)]
\item
 $q+\psi(\xi)$, $q+\psi(\eta)$ and $q+\psi(\eta')$ do not assume values in $(-\infty,0]$;
\item
$\xi-\eta\neq 0$ and  $\xi-\eta-\eta'\neq 0$;
\item
the oscillating factors become fast decaying ones.
\end{enumerate}
If $\nu\in [1,2]$ or $\nu\in (0,1)$ and $\mu=0$, then
 \bbe\label{cone_nupr}
     \Re \psi(\xi)\ge c_{\psi;\infty} |\xi|^{\nu}-C_\psi,\quad \forall\, \xi\in i(\mum,\mup)+(\cC_{\gampr,\gappr}\cup\{0\}),
     \ee
     where $C_\psi, c_{\psi;\infty}>0$ are independent of $\xi$. The following lemma implies that one can construct the contours so that the conditions (i)-(iii) are satisfied.
     \begin{lem}\label{lem:cones_Brom} (\cite[Lemma 4.1]{EfficientLevyExtremum}) Let the characteristic exponent $\psi$ of a SINH-regular process satisfy \eq{cone_nupr}.
     Then there exist $\om_\ell\in (0,\pi/2)$ and $c, \sg>0$ such that for all 
     $q\in \sg+\cC_{\pi/2+\om_\ell}$ and  $\xi\in i(\mum,\mup)+(\cC_{\gampr,\gappr}\cup\{0\})$,
    $ |q+\psi(\xi)|\ge c(|q|+|\xi|^\nu).$
     \end{lem}

\subsection{Reduction: the case $a_1\le 0$}\label{ss:reduction_quintuple_a1le0}
Assume first that $X$ is SINH-regular of order $\nu\in [1,2]$ or of order $\nu\in (0,1)$
and $\mu=0$. We substitute \eq{V_ftd_sinh_sinh} and \eq{derV2_sinh} into \eq{eq:main2}:
\beqast
&&\bP[X_T\le a_1, \barX_T\le a_2, \tau_{T}\le t]\\
&=&\int_0^{a_2} dh\, \frac{1}{2\pi i}\int_{\cL^L_{\sg',b'_\ell,\om'_\ell}}dq'\,\frac{e^{q't}}{q'}\frac{1}{2\pi}
\int_{\cL^-_{\om_1,b,\om}}d\eta'\,e^{-i\eta' h}\frac{\phi^{++}_{q'}(\eta')}{i\eta'}\\
&&\cdot\frac{1}{2\pi i}\int_{\cL^L_{\sg,b_\ell,\om_\ell}}dq\,\frac{e^{qT}}{q}\frac{1}{(2\pi)^2}\int_{\cL^-_{\om^-_1,b^-,\om_-}}d\eta\,e^{-ih\eta}\phi^{++}_q(\eta)  \int_{\cL^+_{\om^+_1,b^+,\om_+}}d\xi\, \frac{e^{i\xi(h-a_1)}\phi^{--}_q(\xi)}{-i\xi}.
\eqast
The sextuple integral being absolutely convergent\footnote{The proof is similar to the proof in \cite{EfficientLevyExtremum} of the absolute convergence
of the integral on the RHS of \eq{V2_sinh} but messier since the bounds are needed for a function of 6 variables  not 3.}, we apply Fubini's theorem and integrate w.r.t. $h$ first. Since
$
\int_0^{a_2} dh\, e^{ih(\xi-\eta-\eta')}=(e^{ia_2(\xi-\eta-\eta')}-1)/(i(\xi-\eta-\eta')),
$ the result is
\bbe\label{sinh_a1le0}
\bP[X_T\le a_1, \barX_T\le a_2, \tau_{T}\le t]=W(a_1,T,t;a_2)-W(a_1,T,t;0),
\ee
where 
\beqast
W(a_1,T,t;h)&=&\frac{1}{2\pi i}\int_{\cL^L_{\sg,b_\ell,\om_\ell}}dq\,\frac{e^{qT}}{q}\frac{1}{(2\pi)^2}
\int_{\cL^-_{\om^-_1,b^-,\om_-}}d\eta\,e^{-ih\eta}\phi^{++}_q(\eta)  \int_{\cL^+_{\om^+_1,b^+,\om_+}} \frac{e^{i\xi(h-a_1)}\phi^{--}_q(\xi)d\xi\,}{\xi}\\
&&\cdot \frac{1}{2\pi i}\int_{\cL^L_{\sg',b'_\ell,\om'_\ell}}dq'\,\frac{e^{q't}}{q'}\frac{1}{2\pi}
\int_{\cL^-_{\om_1,b,\om}}d\eta'\,e^{-i\eta' h}\frac{\phi^{++}_q(\eta')}{i\eta'(\xi-\eta-\eta')}.
\eqast
The calculations in the algorithm in Section \ref{ss:algo_SINH} are arranged according the following representation:
\bbe\label{Wb}
W(a_1,T,t;h)=\frac{1}{(2\pi)^3}\Im\int_{\cL^-_{\om^-_1,b^-,\om_-}}d\eta\,e^{-ih\eta}
 \int_{\cL^+_{\om^+_1,b^+,\om_+}}d\xi\,\frac{e^{i(h-a_1)\xi}}{\xi}S_2(\eta,\xi)S_3(h;\eta,\xi),
\ee
where 
\beqa\label{S1}
S_1(\eta')&=&\frac{1}{2\pi i}\int_{\cL^L_{\sg',b'_\ell,\om'_\ell}}dq'\,\frac{e^{q't}}{q'}\frac{\phi^{++}_{q'}(\eta')}{\eta'},
\eqa
\beqa\label{S2}
S_2(\eta,\xi)&=&\frac{1}{2\pi i}\int_{\cL^L_{\sg,b_\ell,\om_\ell}}dq\,\frac{e^{qT}}{q}\phi^{++}_q(\eta)\phi^{--}_q(\xi),\\\label{S3b}
S_3(h;\eta,\xi)&=&\int_{\cL^-_{\om^-_1,b^-,\om_-}}d\eta'\,\frac{e^{-ih\eta'}S_1(\eta')}{\xi-\eta-\eta'}.
\eqa

\subsection{Reduction: the case $a_1> 0$}\label{ss:reduction_quintuple_a1>0}
Since
\[
\bP[X_T\le a_1, \barX_T\le a_2, \tau_{T}\le t]=\bP[X_T\le 0, \barX_T\le a_2, \tau_{T}\le t]
+\bP[0<X_T\le a_1, \barX_T\le a_2, \tau_{T}\le t]
\]
and the first term on the RHS is calculated in the preceding subsection (set $a_1=0$), it remains to calculate the second term.
We have $\bP[0<X_T\le a_1, \barX_T\le a_2, \tau_{T}\le t]=\bP[0<X_T\le \min\{a_1,a_2\}, \barX_T\le a_2, \tau_{T}\le t]$,
therefore, it suffices to consider the case
$a_1\le a_2$. Then
\beqast
&&\bP[0<X_T\le a_1, \barX_T\le a_2, \tau_{T}\le t]\\
&=&\int_0^{a_1}da\,\int_a^{a_2}dh\,V_{f.t.d.}(h,t)\frac{\dd^2}{\dd a\dd h}V_2(a,h,T)\\
&=&\int_0^{a_1}da\,\int_a^{a_2}dh\,\frac{1}{2\pi i}\int_{\cL^L_{\sg',b'_\ell,\om'_\ell}}dq'\,\frac{e^{q't}}{q'}\frac{1}{2\pi}
\int_{\cL^-_{\om_1,b,\om}}d\eta'\,e^{-i\eta' h}\frac{\phi^{++}_{q'}(\eta')}{i\eta'}\\
&&\cdot\frac{1}{2\pi i}\int_{\cL^L_{\sg,b_\ell,\om_\ell}}dq\,\frac{e^{qT}}{q}\frac{1}{(2\pi)^2}\int_{\cL^-_{\om^-_1,b^-,\om_-}}d\eta\,e^{-ih\eta}\phi^{++}_q(\eta)  \int_{\cL^+_{\om^+_1,b^+,\om_+}}d\xi\, e^{i\xi(h-a_1)}\phi^{--}_q(\xi).
\eqast
Similarly to the proof in \cite{EfficientLevyExtremum} of the absolute convergence
of the integral on the RHS of \eq{V2_sinh}, one can derive a bound for the absolute value of the integrand via a
positive function of class $L_1$, hence, Fubini's theorem is applicable. 
 We use
 \beqast
 \int_0^{a_1}da\,e^{-ia\xi}\int_a^{a_2}dh\, e^{ih(\xi-\eta-\eta')}
 &=&
 \int_0^{a_1}da\,\frac{e^{ia_2(\xi-\eta-\eta')}-e^{ia(\xi-\eta-\eta')}}{i(\xi-\eta-\eta')}e^{-ia\xi}\\
 &=&\frac{e^{i(a_2-a_1)\xi-ia_2(\eta+\eta')}}{i(\xi-\eta-\eta')(-i\xi)}-\frac{e^{ia_2\xi-ia_2(\eta+\eta')}}{i(\xi-\eta-\eta')(-i\xi)}\\
 &&-\frac{e^{-ia_1(\eta+\eta')}}{i(\xi-\eta-\eta')(-i(\eta+\eta'))}+\frac{1}{i(\xi-\eta-\eta')(-i(\eta+\eta'))}
 \eqast
to obtain
\beqa\label{sinh_a1>0}
\bP[0<X_T\le a_1, \barX_T\le a_2, \tau_{T}\le t]&=&W_1(a_2,T,t;a_2-a_1)-W_1(a_2,T,t;a_2)\\\nonumber
&&-W_2(a_1,T,t)+W_2(0,T,t),
\eqa
where 
\beqast
W_1(a_2,T,t;h)&=&
\frac{1}{2\pi i}\int_{\cL^L_{\sg,b_\ell,\om_\ell}}dq\,\frac{e^{qT}}{q}\frac{1}{(2\pi)^2}
\int_{\cL^-_{\om^-_1,b^-,\om_-}}d\eta\,e^{-ia_2\eta}\phi^{++}_q(\eta)  \int_{\cL^+_{\om^+_1,b^+,\om_+}}d\xi\, \frac{e^{ih\xi}\phi^{--}_q(\xi)}{\xi}\\
&&\cdot \frac{1}{2\pi i}\int_{\cL^L_{\sg',b'_\ell,\om'_\ell}}dq'\frac{e^{q't}}{q'}\frac{1}{2\pi}
\int_{\cL^-_{\om_1,b,\om}}d\eta'\,e^{-i a_2\eta'}\frac{\phi^{++}_q(\eta')}{i\eta'(\xi-\eta-\eta')},\eqast
\beqast
W_2(h,T,t)&=&
\frac{1}{2\pi i}\int_{\cL^L_{\sg,b_\ell,\om_\ell}}dq\,\frac{e^{qT}}{q}\frac{1}{(2\pi)^2}
\int_{\cL^-_{\om^-_1,b^-,\om_-}}d\eta\,e^{-ih\eta}\phi^{++}_q(\eta)  \int_{\cL^+_{\om^+_1,b^+,\om_+}} d\xi\,\phi^{--}_q(\xi)\\
&&\cdot \frac{1}{2\pi i}\int_{\cL^L_{\sg',b'_\ell,\om'_\ell}}dq'\frac{e^{q't}}{q'}\frac{1}{2\pi}
\int_{\cL^-_{\om_1,b,\om}}d\eta'\,e^{-ih\eta'}\frac{\phi^{++}_q(\eta')}{i\eta'(\xi-\eta-\eta')(\eta+\eta')}.\\
\eqast
The calculations in the algorithm below are arranged according to the following representation:
\beqast\label{W1b}
W_1(a_1,T,t;h)&=&\frac{1}{(2\pi)^3}\Im\int_{\cL^-_{\om^-_1,b^-,\om_-}}d\eta\,e^{-ia_2\eta}
 \int_{\cL^+_{\om^+_1,b^+,\om_+}}d\xi\,\frac{e^{ih\xi}}{\xi}S_2(\eta,\xi)S_{31}(a_2;\eta,\xi)\\\label{W2b}
W_2(h,T,t)&=&\frac{1}{(2\pi)^3}\Im\int_{\cL^-_{\om^-_1,b^-,\om_-}}d\eta\,e^{-ih\eta}
 \int_{\cL^+_{\om^+_1,b^+,\om_+}}d\xi\,S_2(\eta,\xi)S_{32}(h;\eta,\xi),
\eqast
where 
\beqast
S_{31}(a_2;\eta,\xi)&=&\int_{\cL^-_{\om^-_1,b^-,\om_-}}d\eta'\,\frac{e^{-ia_2\eta'}S_1(\eta')}{\xi-\eta-\eta'},
\\
S_{32}(h;\eta,\xi)&=&\int_{\cL^-_{\om^-_1,b^-,\om_-}}d\eta'\,\frac{e^{-ih\eta'}S_1(\eta')}{(\eta+\eta')(\xi-\eta-\eta')}.
\eqast

\subsection{Evaluation using GWR algorithm and/or summation by parts}\label{ss:eval_GWR_summ}
Formally, we can use flat contours of integration and summation by parts w.r.t.  all variables. If $\nu\in [1,2]$ or $\nu\in(0,1)$ and $\mu=0$, then it follows from the bounds for the derivatives of the Wiener-Hopf factors that the rate of decay increases with each application of the summation by parts, in each integral (see Lemma \ref{lem:der_WHF_q_SINH}). Formally, we can also apply GWR algorithm to evaluate
the integral w.r.t. $q'$ and $q$ but it is advisable not to apply the GWR-algorithm in both integrals if double precision arithmetic is used. The reason 
is that even one application of  GWR algorithm introduces an error of the order of E-07, at least, and to satisfy the error tolerance of the order of E-06,
the values of the integrand must be  evaluated with the accuracy E-13.

\section{Algorithms of SINH-method and numerical examples}\label{s:numer}

\subsection{Algorithms}\label{ss:algo_SINH}
  We formulate the algorithms for the evaluation of the RHS of \eq{eq:main2} using SINH-method
and assuming that $\nu\in [1,2]$ or $\nu\in (0,1)$ and $\mu=0$.
  In this case, $\phi^{\pm\pm}_q=\phi^\pm_q$, and the same  contours in the $\xi$-and $\eta$-spaces can be used for  for the evaluation   of the Wiener-Hopf factors and in the main formulas.
  The deformations must be in the agreement explained in
 Sect. \ref{ss: sinhU2}. 
 Some of the blocks below are  algorithms  borrowed from  \cite{Contrarian,EfficientLevyExtremum}
 for the evaluation of $V_{n.t.}(a_2,t)$ and $V_2(a_1,a_2,T)$. 

 \subsubsection{Algorithm I. Sinh-acceleration is applied to all integrals}\label{ss:algo_I}
 \begin{enumerate}[Step I.]
\item 
Choose the sinh-deformation in the Bromwich integrals  and grids for the simplified trapezoid rule: $\vec{y}=\ze_\ell*(0:1:N_\ell)$,
$\vec{q}=\sg_\ell+i*b_\ell*\sinh(i*\om_\ell+\vec{y})$, and $\vec{y'}=\ze'_\ell*(-N'_\ell:1:N'_\ell)$,
$\vec{q'}=\sg'_\ell+i*b'_\ell*\sinh(i*\om'_\ell+\vec{y'})$.
Unless $t<<T$, the deformation can be the same
for each integral but even in this case, it may be advantageous to choose $\ze'_\ell<\ze_\ell$ and $N'_\ell\ze'_\ell>N_\ell\ze_\ell$. 

\item
Calculate  the (normalized) derivatives \[
\vec{der_\ell}=b_\ell*\cosh(i*\om_\ell+\vec{y}),\ \vec{der'_\ell}=b'_\ell*\cosh(i*\om'_\ell+\vec{y'})\]
and arrays of weigths \[
\cQ=(\ze_\ell/\pi)*\vec{der_\ell}(q).*\exp(T*\vec{q})./\vec{q},\quad
 \cQ'=(\ze_\ell/(2*\pi)*
\vec{der'_\ell}.*\exp(T*\vec{q'})./\vec{q'}t.\]
Denote the elements of $\cQ$ and $\cQ'$ by $\cQ(q)$ and $\cQ'(q')$, and reassign $\cQ(q_0)=\cQ(q_0)/2$.
\item
Choose  the sinh-deformations and grids for the simplified trapezoid rule on $\cL^\pm$: $\vec{y^\pm}=\ze^\pm*(-N^\pm:1:N^\pm)$,
$\vec{\xi^\pm}=i*\om_1^\pm+ b^\pm*\sinh(i*\om^\pm+\vec{y^\pm})$.\\  Calculate $\vec{\psi^\pm}=\psi(\vec{\xi^\pm})$ and 
$\vec{der^\pm}=b^\pm*\cosh(i*\om^\pm+\vec{y^\pm}).
$
\item
{\em Grids for evaluation of the Wiener-Hopf factors $\phi^\pm_q(\xi)$.} Choose longer and finer grids for the simplified trapezoid rule on $\cL^\pm_1$: $\vec{y^\pm_1}=\ze_1^\pm*(-N^\pm_1:1:N^\pm_1)$,
$\vec{\xi^\pm_1}:=i*\om^\pm_1+ b^\pm_1*\sinh(i*\om_1^\pm+\vec{y^\pm_1})$.  Calculate $\psi^\pm_1:=\psi(\vec{\xi^\pm_1})$ and 
$\vec{der^\pm_1}:=b^\pm_1*\cosh(i*\om_1^\pm+\vec{y^\pm_1}).
$
Note that it is unnecessary to use sinh-deformations different from the ones on Step III but it is advisable to write a program allowing
for different deformations in order to be able to control errors of each block of the program separately.
\item
Calculate 2D arrays 
\beqast
D^{+-}_1&:=&1./(\mathrm{conj}(\vec{\xi^-_1})'*\mathrm{ones}(1,2*N^++1)-\mathrm{ones}(2*N^-_1+1,1)*\vec{\xi^+})),\\
D^{-+}_1&:=&1./(\mathrm{conj}(\vec{\xi^+_1})'*\mathrm{ones}(1,2*N^-+1)-\mathrm{ones}(2*N^+_1+1,1)*\vec{\xi^-})).\\
\eqast
\item
 For $q\in \vec{q}$, calculate $\vec{\phipq}=\phipq(\vec{\xi^+})$ and $\vec{\phimq}=\phimq(\vec{\xi^-})$:
\beqast
\vec{\phipq}&:=&\exp((\ze^-_1*i/(2*\pi))*\vec{\xi^+_1}.*((\log(1+\psi^-_1/q)./\vec{\xi^-_1}.*\vec{der^-_1})*D^{-+}_1)),\\
\vec{\phimq}&:=&\exp(-(\ze^+_1*i/(2*\pi))*\vec{\xi^-_1}.*((\log(1+\psi^+_1/q)./\vec{\xi^+_1}.*\vec{der^+_1})*D^{+-}_1)),
\eqast
and then calculate and store 2D arrays
\bbe\label{phimppm}
\vec{\phi^+_{q',-}}:=1./(1+\psi^-./q)./\vec{\phimq},\ \vec{\phi^-_{q,+}}:=1./(1+\psi^+./q)./\vec{\phipq}.
\ee
\item
  Calculate 1D and 2D arrays 
 \beqast
 S_1&=& \sum_{q'}\cQ'(q')*(\vec{\phi^+_{q',-}}./\vec{\xi^-}).\\
 S_2&=&\sum_q \cQ(q)*\mathrm{conj}\,(\vec{\phi^+_{q,-}})'* \vec{\phi^-_{q,+}}.\\
  \eqast
 \item
 In the double cycle in $j=-N^-:1:N^-$, $k=-N^+:1:N^+$, for $h=0,a_2$, calculate the entries $S_{3,h}(j,k)$ of 
 2D arrays 
 \[
 S_{3,h}(j,k)=(1/(\xi^+_k-\xi^-_j))*\mathrm{sum}\,(1./(\xi^+_k-\xi^-_j-\vec{\xi^-}).*S_1.*\exp(-i*h*\vec{\xi^-})).
 \] 
 \item
 For $h=0,a_2$, calculate 
  \beqast
  W_{h}&=&((\ze^+(\ze^-)^2)/(2*\pi)^3)*\mathrm{imag}\, \left((\vec{der^-}.*\exp(-i*h*\vec{\xi^-}))...\right.\\
  &&\left.*(S_2.*S_3)*\mathrm{conj}\,(\vec{der^+}.*\exp(i*(h-\min\{a_1,0\})*\vec{\xi^+})'\right),
  \eqast
 and then $V_0=W_{a_2}-W_{0}$.
 \item
 If $a_1\le 0$, $\bP[X_T\le a_, \barX_T\le a_2, \tau_{T}\le t]=V_0$.
 \item If $a_1>0$, $\bP[X_T\le a_, \barX_T\le a_2, \tau_{T}\le t]=V_0+V_1$,
 where $V_1$ is calculated as follows.
 \item
 Reassign $a_1=\min\{a_1,a_2\}$, and, in the
  double cycle in $j=-N^-:1:N^-$, $k=-N^+:1:N^+$, calculate the entries  of 
 2D arrays 
 \beqast
 S_{31}(j,k)&=&\mathrm{sum}(S_1.*\exp(-(i*a_2)*\vec{\xi^-}).*(\vec{der^-}./(\xi^+_k-\xi^-_j-\vec{der^-}))),\\
                    S_{32;a1}(j,k)&=&\mathrm{sum}(S_1.*\exp(-(i*a_1)*\vec{\xi^-}).*
                        (\vec{der^-}./(\xi^+_k-\xi^-_j-\vec{der^-})./(\xi^-_j-+\vec{\xi^-})));\\
                    S_{32;0}(j,k)&=&\mathrm{sum}(S_1.*(\vec{der^-}./(\xi^+_k-\xi^-_j-\vec{\xi^-})./(\xi^-_j-+\vec{\xi^-})));
\eqast 
\item
For $h=a_2-a_1, a_2$, calculate
\[
W_{1;h}=\mathrm{imag}((\vec{der^-}.*\exp(-(i*a_2)*\vec{\xi^-}))*(S_2.*S_{31})
                *\mathrm{conj}(\vec{der^+}.*\exp(i*h*\vec{\xi^+})./\vec{\xi^+})').
                \]
\item
For $h=0,a_1$, calculate 
\[
W_{2;h}=\mathrm{imag}((\vec{der^-}.*\exp(-(i*a_1)*\vec{\xi^-}))*(S_2.*S_{32;h})
                *\mathrm{conj}(\vec{der^+})').
                \]
\item
 Set            $V_1=((\ze^+(\ze^-)^2)/(2*\pi)^3)*(W_{1;a_2-a_1}-W_{1;a_2}-W_{2;a_1}+W_{2;0}).$            
\end{enumerate}

 \subsubsection{Algorithm II. Sinh-acceleration is applied to all integrals but the one w.r.t. $q$}\label{ss:algo_II}
 The weights $\cQ$ in the numerical procedures
 are the weights in the GS algorithm, and the GWR algorithm is applied at the very last step, after 4D integrals are evaluated. 
  
 \begin{rem}\label{rem:intGS} {\rm Note that unless high precision arithmetic is used, it is not safe to apply   GWR-algorithm
 to the integrals w.r.t. $q'$ because of the sizable error of calculations.}
 \end{rem} 
 
\subsubsection{Algorithm III. Summation by parts is applied}\label{ss:algo_III}
The infinite trapezoid rule can be used to any of the integrals without applying the conformal deformation technique;
the number of terms can be decreased using the summation by parts. The corresponding matrix representations in the algorithm
such as $a*A*b'$, where $a,b$ are 1D arrays and $A$ is 2D array, should be replaced by the explicit double summation (and summation by parts in each infinite sum).

\subsection{Numerical examples}\label{sss:numer}
We consider the same example as shown in Fig.~\ref{meshes} in Introduction  but calculate the values
at a sparser grid. The results are in Table \ref{table1}. The process is KoBoL with the parameters $\nu=1.2,0.8,0.5,0.3, \mu=0, \lp=1, \lm=-2$,
the second instantaneous moment $m_2=\psi^{\prime\prime}(0)=0.1$ fixes  $c=c(\nu,m_2,\lp,\lm)$.
The errors and CPU time of SINH, DISC-SINH and DISC-GWR methods for different choices of the parameters of the numerical schemes are in Tables~\ref{table2}-\ref{table4}. We use different grids for SINH-based and GWR-based methods; if GWR algorithm is used then one cannot hope to obtain significantly more accurate results even for the
joint probabilities of $(X_T,\barX_T)$. See \cite{EfficientLevyExtremum}. We do not show the results
of SINH-GWR algorithm: the errors are larger than the errors of DISC methods.

 In the numerical examples, we choose the parameters of the scheme as follows.
  If the sinh-acceleration is applied to the Bromwich integral, we set 
 $\om_\ell=\pi/10$, and choose $\sg_\ell$, $b_\ell$, $b^\pm$, $\om^\pm_1, \om^\pm$ following the recommendation in \cite{EfficientLevyExtremum}
 for the evaluation of $V_2(a_1,h,T)$.  If GWR algorithm is used, we follow the recommendations in
\cite{EfficientLevyExtremum}. The errors of the benchmark values shown in Table~\ref{table1} are estimated as differences between values
 obtained using the algorithm with $k_{\om_\ell}\om_\ell$ and $k_\om \om$ in place of $\om_\ell$ and $\om_\ell$, where $k_\om\in \{0.9,1\}$ and $k_{\om_\ell}\in \{0.8,0.9,1\}$. 
 The general recommendations for the choices of grids
 in \cite{EfficientLevyExtremum} are based on approximate bounds for the Hardy norms. The bounds being rough, 
 the resulting prescriptions are not very accurate, in the case of the exterior Bromwich integral especially.
 We use the general prescriptions for the error tolerances $\eps=10^{-N^{e}_{WHF}}$ and $\eps=10^{-N^e}$ 
 for evaluation of integrals in the formulas for the Wiener-Hopf factors and quintuple integrals, respectively,
 and adjust the steps and the number of terms decreasing steps by the factor of 1.2 and increasing the number of terms by $1.2^2$,
 with the exception of $\ze_\ell$ and $N_\ell$. Apparently, the ad-hoc bound for the Hardy norm of the integrand in the Bromwich integral (the integrand is a quadruple integral)
 is too inaccurate, and we use the factors  3 and 2.5, respectively, to achieve the accuracy shown in  Table~\ref{table2}.
 For the convenience of the presentation, we choose equal $\ze^+=\ze^-$, $N^+=N^-$, $\ze^+_1=\ze^-_1$, $N^+_1=N^-_1$.
 
 The accuracy and CPU time can be improved if the explicit recommendation for the choice of the parameters of the numerical scheme is used for each point $(a_1,a_2)$ separately; having in mind applications to the Monte-Carlo simulations in vein of
 \cite{SINHregular,ConfAccelerationStable}, we use the same deformations and grids for all $(a_1,a_2)$. \begin{table}
\caption{\small Joint probabilities   (rounded)  calculated using SINH-method.
}
 {\tiny
\begin{tabular}{c|ccccc}
\hline\hline
 & 0.025 & 0.05  & 0.075 & 0.1 & 0.125\\\hline
& & & $\nu=1.2$ & &\\
$-0.10$ & 0.09485687438 & 0.10479444335 & 0.10678805405 & 0.10727897343 & 0.10742368440
  \\
$-0.05$ & 0.17396324242 & 0.19336016343 & 0.19700962934 & 0.19784558698 & 0.19807756974\\

  $0$ & 0.27360405736	 &	0.31628178239	&	0.32433174442	&	0.32601570450	&	0.32644136471\\
 $0.05$ &0.29459120084&	0.37284195852&	0.39457828203	& 	0.39890812914&	0.39987799975\\
 $0.10$ & 0.29459120084	&	0.37284195852&	0.40227057148	& 	0.41377943268&	0.41695149176\\\hline
& & & $\nu=0.8$ & &\\
$-0.10$ & 0.058582602 & 0.059642703 & 0.059836034 & 0.059893044 & 0.059914503\\
$ 0.05 $ & 0.109632518 & 0.111393671 & 0.111690407 & 0.111773323 & 0.111803328\\
$ 0$ & 0.318065334 & 0.322186226 & 0.322739900 &     0.322876977 & 0.322922943\\
$0.05$ & 0.388768010 & 0.411867555 & 0.414270173 & 0.414587103 & 0.414672505\\
$0.10$ &   0.388768010 & 0.411867555 & 0.417654511 & 0.419778209 & 0.420143628\\\hline
& & & $\nu=0.5$ & &\\
$-0.10$ & 0.0393672 & 0.0395104 & 0.0395447 & 0.0395573 & 0.0395630\\
$-0.05 $ & 0.0648411 & 0.0650523 & 0.0651001 & 0.0651170 & 0.0651243\\
$ 0 $ & 0.3196310 & 0.3200335 & 0.3201102 & 0.3201350 & 0.3201452\\
$0.05$ & 0.4133573 & 0.4182210 & 0.4185020 & 0.4185487 & 0.4185650\\
$0.10$ & 0.4133573 & 0.4182210 & 0.4196328 & 0.4202368 & 0.4202946\\\hline
& & & $\nu=0.3$ & &\\
$-0.10$ & 0.029582 & 0.029619 & 0.029630 & 0.029634 & 0.029637\\
$-0.05 $ & 0.044663 & 0.044715 & 0.044729 &	0.044735 &	0.044738\\
$0$ & 0.313754&	0.313840&	0.313861 &	0.313870 &	0.313873\\
$0.05$ & 0.411966 &	0.413583 &	0.413642	& 0.413655 &	0.413661\\
$0.10$ & 0.411966 &	0.413583&	0.414132 &	0.414393 &	0.414408\\\hline
\end{tabular}
}
\begin{flushleft}{\tiny 
Upper row: $a_2$. Left column: $a_1$. The total CPU time for evaluation at 25 points, in seconds, and range of errors:
$\nu=1.2$: 813.9, [4*E-11, E-10];  $\nu=0.8$: 1,182, [4*E-09, 3*E-08]; $\nu=0.5$: 779.0, [E-06,9*E-06];
     $\nu=0.3$: 764.5, [2*E-05, 2*E-04].}
 \end{flushleft}

\label{table1}
 \end{table}

 \begin{table}
\caption{\small  Parameters of SINH-algorithm (steps are rounded), CPU time and range of errors}

 {\tiny
\begin{tabular}{c|cc|cc|cc|cc| cc|c|c}
\hline\hline
$\nu$ & $N^e$ & $N^{e}_{WHF}$ & $\ze_\ell$ & $N_\ell$ & $\ze'_\ell$ & $N'_\ell$ & $\ze$ & $N^\pm$ & $\ze_1$ & $N^\pm_1$ 
& Time (sec.) & Range of Errors 
\\\hline 
$1.2$ & 11 & 13 & 0.07 & 200 & 0.07 & 221 & 0.099 & 335 & 0.085 & 455 & 813.9 & $ [4*E-11, E-10]$ \\
& 8    & 10 & 0.09 & 145 & 0.09   & 162 & 0.13   & 188 & 0.10 & 281 & 193.8 & $ [3*E-07, E-06]$\\ 
& 5.5 & 7  & 0.12  & 114 & 0.12  & 127 & 0.18 & 89   & 0.15 & 149 & 17.5 & $[2*E-05, 6*E-05]$ \\\hline
$0.8$ & 12 & 14 & 0.06 & 237 & 0.06 & 262 & 0.10 & 380 & 0.09 & 496 & 1,182 & $ [4*E-09, 3*E-08]$\\
& 11 & 13 & 0.07 & 200 & 0.07 & 221 & 0.12 & 279 & 0.10 & 380 & 476.5 & $ [2*E-08, 8*E-08]$\\
& 8 & 10 & 0.09 & 145 & 0.09 & 162 & 0.16 & 157 & 0.13 & 235 & 127.9 & $ [2*E-05, 2*E-04]$\\\hline
$0.5$ & 12 & 14 & 0.065 & 218 & 0.065 & 241 & 0.11 & 327 & 0.095 & 436   & 779.0 & $[E-06,9*E-06]$\\ 
& 8& 10& 0.09 & 145 & 0.09 &162  & 0.16 & 157 & 0.13 & 235 & 127.8 & $[0.004,0.006]$
 \\\hline
$0.3$ & 12 & 14 & 0.065 & 218 & 0.065 & 241 & 0.11 & 327 & 0.095 & 436&764.5 &$ [2*E-05, 2*E-04]$\\
& 9 & 11 & 0.08 & 163 & 0.08 & 181 & 0.14 & 194 & 0.12 & 279 & 215.0 & $[0.0013, 0.015]$\\\hline
 
\end{tabular}
}
\label{table2}
 \end{table} 
 
  \begin{table}
\caption{\small  Parameters of DISC-SINH-algorithm (steps are rounded), CPU time and range of errors}

 {\tiny
\begin{tabular}{c|cc|c|cc|cc| cc|c|c}
\hline\hline
$\nu$ & $N^e$ & $N^{e}_{WHF}$ & $\De h$ &$\ze_\ell$ & $N_\ell$ & $\ze$ & $N^\pm$ & $\ze_1$ & $N^\pm_1$ 
& Time (sec.) & Range of Errors 
\\\hline 
$1.2$ & 8 &  10 & $3.13*10^{-5}$ & 0.09 & 145 &  0.133 & 188 & 0.101 & 281 & 611.1 & $ [4*E-09, 2*E-08]$ \\
& 8 & 10 &  $1.56*10^{-5}$ & 0.09 & 145 &  0.133 & 188 & 0.101 & 281 & 1,202 & $ [2*E-09, 6*E-09]$ \\\hline
$0.8$ & 5.5 & 7 & $6.25*10^{-5}$ & 0.12 & 102 & 0.22 & 82 & 0.18 & 124 & 97.6 &  $[2*E-06, 4*E-06]$\\
& 5.5 & 7 & $3.13*10^{-5}$ &  0.12 & 102 & 0.22 & 82 & 0.18 & 124 & 185.0 & $[3*E-0.7, 2*E-06]$\\
& 5.5 & 7 & $1.56*10^{-5}$ &   0.12 & 102 & 0.22 & 82 & 0.18 & 124 & 367.7 & $[2*E-08, 4*E-07]$\\
& 8 & 10 & $1.56*10^{-5}$ & 0.09 & 145 & 0.16 & 157 & 0.13 & 235 & 1,092 & $[2*E-08, 4*E-07]$\\
& 8 & 10 & $7.81*10^{-6}$ & 0.09 & 145 & 0.16 & 157 & 0.13 & 235 & 2,159 & $[6*E-09, 3*E-07]$\\\hline
$0.5$ &  8 & 10 & $6.25*10^{-5}$ & 0.09 & 145 & 0.16 & 157 & 0.13 & 235 & 253.6 & $[1.5*E-05, 1.5*E-04]$ \\
&  8 & 10 & $6.25*10^{-5}$& 0.09 & 145 & 0.16 & 157 & 0.13 & 235 & 503.4 &$[2*E-05, 1.1*E-04]$  \\\hline
$0.3$ & 8 & 10 & $3.13*10^{-5} $ & 0.09 & 145 & 0.16 & 157 &0.13 & 235 & 560.2 & $[0.004, 0.006]$\\
& 8 & 10 & $1.57*10^{-5}$ &  0.09 & 145 & 0.16 & 157 &0.13 & 235 & 1,110 & $[0.0002, 0.005]$\\
& 8 & 10 & $ 7.81*10^{-6}$ & 0.09 & 145 & 0.16 & 157 &0.13 & 235 & 2,425 & $ [3*E-05,0.005]$\\\hline

\end{tabular}
}
\label{table3}
 \end{table} 
 
 \begin{table}
\caption{\small  Parameters of DISC-GWR algorithm (steps are rounded), CPU time and range of errors}

 {\tiny
\begin{tabular}{c|cc|c|cc|cc|c|c}
\hline\hline
$\nu$ & $N^e$ & $N^{e}_{WHF}$ & $\De h$ & $\ze$ & $N^\pm$ & $\ze_1$ & $N^\pm_1$ 
& Time (sec.) & Range of Errors 
\\\hline 
$1.2$ & 5.5 &  7& $2.5*10^{-4}$ &     0.18 & 99& 0.15 & 149 & 38.6 & $ [6*E-07, 3*E-06]$ \\
& 8 &  10 & $1.56*10^{-5}$ &     0.133 & 188 & 0.101 & 281 & 64.8 & $ [3*E-07, 3*E-06]$ \\\hline
$0.8$ & 5.5 &  7& $3.13*10^{-5}$ &  0.18 & 99& 0.15 & 149 & 30.5 &  $[2*E-07, 3*E-06]$\\
& 8 & 10 & $ 1.56*10^{-5}$ & 0.133 & 188 & 0.101 & 281 & 93.4 & $ [E-07,2*E-06]$\\
& 8 & 10 & $7.81*10^{-6}$ & 0.133 & 188 & 0.101 & 281 & 191.3 & $ [2*E-08,2*E-06]$\\\hline
$0.5$ &  8 & 10 & $6.25*10^{-5}$  & 0.16 & 157 & 0.13 & 235 & 24.6 & $[6*E-05, 2*E-04]$ \\
&  8 & 10 & $6.25*10^{-5}$&   0.16 & 157 & 0.13 & 235 & 44.2 &$[7*E-06, 3*E-04]$  \\\hline
$0.3$ &  8 & 10 & $3.13*10^{-5}$ &  0.16 & 157 &0.13 & 235 & 50.7 & $[0.0004,0.006]$\\
& 8 & 10 & $ 1.56*10^{-5}$ & 0.16 & 157 &0.13 & 235 & 107.0 & $[0.0002,0.005]$\\
& 8 & 10 & $7.81*10^{-6}$   & 0.16 & 157 &0.13 & 235 & 267.8 & $[3*E-05,0.005]$
 \\\hline

\end{tabular}
}
\label{table4}
 \end{table}

\section{Conclusion}\label{s:concl}
In the paper, we designed two efficient analytical methods for evaluation of the joint cpdf $V$ of a L\'evy process $X$ with exponentially decaying tails, its supremum process,  and the first time at which $X$ attains its supremum.
The joint cpdf is represented as the Riemann-Stieltjes integral. The integrand is expressed in terms of the probability distribution of the supremum process and 
joint probability distribution function of the process and its supremum process; the latter are evaluated using the methods
developed in \cite{Contrarian,EfficientLevyExtremum}. The first method of the paper is straightforward: the integral is evaluated using an analog of the
 trapezoid rule. However, in the majority of popular L\'evy models, the prices of barrier options are irregular at
 boundary unless the Brownian motion component is not small. The irregular behavior of prices of barrier options at the boundary implies that an interpolation procedure
 implied by a trapezoid-like rule is inaccurate unless a very fine grid is used, and the convergence of the method is very slow,
 which we illustrate with numerical examples.
 A practically useful error bound is essentially impossible to derive. We estimate the errors using benchmarks produced
 by the second method in the paper, which is more accurate albeit slower. We calculate explicitly the Laplace-Fourier transform of $V$ w.r.t. all arguments,
 apply the inverse transforms, and reduce the problem to the evaluation of the sum of 5D integrals (two Laplace inversions and three Fourier inversions). 
The integrals can be evaluated using the summation by parts in the infinite trapezoid rule and  simplified trapezoid rule; the inverse
Laplace transforms can be calculated using the Gaver-Stehfest and Gaver-Wynn-Rho algorithms as well. However, 
if these methods are used, then, typically, high precision arithmetic is needed. 
Under additional conditions on the domain of analyticity of the characteristic exponent,  
the speed of calculations is greatly increased using the conformal deformation technique (sinh-acceleration), and double precision arithmetic allows one to obtain fairly accurate results.
 The program in Matlab running on a Mac with moderate characteristics achieves the precision better than E-07 in 
 a fraction of a second, and the error tolerance of the order of $E-11$ can be satisfied in several seconds per point if
 the result is calculated at several dozen of points in the state space. For a given error tolerance $\eps$, the complexity of the scheme
 is of the order of $(E\ln E)^5$, where $E=\ln(1/\eps)$, and the formal error bounds can be derived. However, simple universal bounds
 and the resulting recommendation for the choice of the parameters of the numerical scheme can be insufficiently accurate or
 lead to an overkill. As in \cite{SINHregular,Contrarian,EfficientLevyExtremum}, we control the error comparing results
 obtained with different conformal deformations; the probability of a random agreement is negligible.
 
  For the methods used in the paper, it is important that the characteristic exponent $\psi$ admits analytic continuation
to a union of a cone and strip containing or adjacent to the real line and enjoy
a regular behavior at infinity. To simplify the presentation, in the paper, we assume that the cone and strip contain the real line.
In the case of stable L\'evy processes different from the Brownian motion, 
the characteristic exponent $\psi$ admits analytic continuation to a cone but not to  any strip and the constructions of the paper
need certain modifications similar to the modifications of the results in \cite{SINHregular,EfficientLevyExtremum}
to the case of stable L\'evy processes in \cite{ConfAccelerationStable,EfficientStableLevyExtremum}.

The algorithms in the paper can be regarded  as  further steps in  a general program of study of the efficiency
of combinations of one-dimensional inverse transforms for high-dimensional inversions systematically pursued by
Abate-Whitt, Abate-Valko \cite{AbWh,AbWh92OR,AbateValko04,AbValko04b,AbWh06} and other authors.
Additional methods can be found in \cite{stenger-book}.
The methods developed in the paper can be used to develop new efficient methods
for Monte Carlo simulations.


\appendix

\section{Technicalities}\label{s:tech}
 \subsection{Infinite trapezoid rule}\label{infTrap} 

Let $g$ be analytic in the strip
$S_{(-d,d)}:=\{\xi\ | \Im\xi\in (-d,d)\}$ and decay at infinity sufficiently fast so that
$\lim_{A\to \pm\infty}\int_{-d}^d |g(i a+A)|da=0,$
and 
\bbe\label{Hnorm}
H(g,d):=\|g\|_{H^1(S_{(-d,d)})}:=\lim_{a\downarrow -d}\int_\bR|g(i a+ y)|dy+\lim_{a\uparrow d}\int_\bR|g(i a+y)|dy<\infty
\ee
is finite. We write $g\in H^1(S_{(-d,d)})$. The integral
$I=\int_\bR g(\xi)d\xi$
can be evaluated using the infinite trapezoid rule
\bbe\label{inftrap}
I\approx \ze\sum_{j\in \bZ} g(j\ze),
\ee
where $\ze>0$. 
The following key lemma is proved in \cite{stenger-book} using the heavy machinery of sinc-functions. A simple proof can be found in
\cite{paraHeston}.
\begin{lem}[\cite{stenger-book}, Thm.3.2.1]
The error of the infinite trapezoid rule admits an upper bound 
\bbe\label{Err_inf_trap}
{\rm Err}_{\rm disc}\le H(g,d)\frac{\exp[-2\pi d/\ze]}{1-\exp[-2\pi d/\ze]}.
\ee
\end{lem}
Once
an approximately bound for $H(g,d)$ is derived, it becomes possible to choose $\ze$ to satisfy the desired error tolerance.

\subsection{Simplified trapezoid rule and summation by parts}\label{ss:summation_by_parts}\label{ss:sum-by-parts}
The infinite sum \eq{inftrap} is truncated replacing $\sum_{j\in\bZ}$ with 
$\sum_{|j|\le N}.$ The rate of decay of the series $\{g(j\ze)\}$ can be significantly increased and the number of terms $N$ sufficient to satisfy a
 given error tolerance decreased 
if the infinite trapezoid rule 
 is of the form 
\[
I(a)=\ze\sum_{j\in \bZ} e^{-i aj\ze}g(j\ze),
\]
where $a\in\bR\setminus 0$, and $g'(y)$ decreases faster than $g(y)$ as $y\to\pm\infty$. 
Indeed, then, by the mean value theorem,  
the finite differences $\De g_j=(\De g)(j\ze)$, where $(\De g)(\xi)=g(\xi+\ze)-g(\xi)$, decay faster 
than $g(j\ze)$ as $j\to\pm\infty$ as well. 

The summation by parts formula is as follows. Let $e^{i a\ze}-1\neq 0$. Then
\[
\ze\sum_{j\in \bZ} e^{-i aj\ze }g(j\ze)=\frac{\ze}{e^{i a\ze}-1}\sum_{j\in \bZ} e^{-i aj\ze }\De g_j.
\]
If additional differentiations further increase the rate of decay of the series as $j\to\pm\infty$, then the summation by part procedure can be iterated:
\bbe\label{e:sum_by_part}
\ze\sum_{j\in \bZ} e^{-i aj\ze }g_j=\frac{\ze}{(e^{i a \ze}-1)^n}\sum_{j\in \bZ} e^{- i aj\ze }\De^n g_j.
\ee
After the summation by parts, the series on the RHS of \eq{e:sum_by_part} needs to be truncated.
The truncation parameter can be chosen using the following lemma.
\begin{lem}\label{lem:err_sum_by_parts}
Let $n\ge 1, N>1$ be  integers, $\ze>0, a\in \bR$ and $e^{i a\ze }-1\neq 0$. 

Let $g^{(n)}$ be continuous and let the function
$\xi\mapsto G_n(\xi,\ze):=\max_{\eta\in [\xi,\xi+n\ze]}|g^{(n)}(\eta)|$ be in $L_1(\bR)$.
Then
\beqa\label{trunc_err_sum_by_parts_pos}
\left|\frac{\ze}{(e^{i a\ze}-1)^n} \sum_{j\ge N} e^{-i aj\ze }\De^n g_j\right|&\le& \left(\frac{\ze}{|e^{i a\ze}-1|}\right)^n\int_{N\ze}^{+\infty}G_n(\xi,\ze)d\xi,
\\\label{trunc_err_sum_by_parts_neg}
\left|\frac{\ze}{(e^{i a\ze}-1)^n} \sum_{j\le -N} e^{-i a j\ze }\De^n g_j\right|&\le& \left(\frac{\ze}{|e^{i a\ze}-1|}\right)^n\int_{-\infty}^{-N\ze}G_n(\xi,\ze)d\xi.
\eqa
\end{lem}
\begin{proof} Using the mean value theorem, we obtain
\[
|(\De^n g)(\xi)|\le \ze \max_{\xi\in [\xi,\xi+\ze]}|(\De^{n-1} g')(\xi)|\le\cdots \le\ze^n\max_{\eta\in [\xi,\xi+n\ze]}|g^{(n)}(\eta)|.
\]
\end{proof}

  \subsection{Gaver-Wynn Rho algorithm}\label{GavWynn} 
The Gaver-Stehfest algorithm approximates  the inverse Laplace transform $V(T)$  of $\tilde V$  by 
\begin{equation}\label{GS31}
V(T, M)=\frac{\ln(2)}{t}\sum_{k=1}^{2M}\zeta_k \tV\left(\frac{k\ln(2)}{T}\right),
\end{equation}
where $M\in \bN$,
\begin{equation}\label{GS32}
\zeta_k(t, M)=(-1)^{M+k}\sum_{j=\lfloor (k+1)/2\rfloor}^{\min\{k,M\}}\frac{j^{M+1}}{M!}\left(\begin{array}{c} M \\ j\end{array}\right)
\left(\begin{array}{c} 2j \\ j\end{array}\right)\left(\begin{array}{c} j \\ k-j\end{array}\right)
\end{equation}
and $\lfloor a \rfloor$ denotes the largest integer that is less than or equal to  $a$. 
 In the paper, as in \cite{paired,Contrarian,EfficientLevyExtremum}, we apply Gaver-Wynn-Rho (GWR) algorithm, which is
  more stable than the Gaver-Stehfest method.
  Given a converging sequence $\{f_1, f_2,
\ldots\}$, Wynn's algorithm estimates the limit $f=\lim_{n\to\infty}f_n$ via $\rho^1_{N-1}$, where $N$ is even,
and $\rho^j_k$, $k=-1,0,1,\ldots, N$, $j=1,2,\ldots, N-k+1$, are calculated recursively as follows:
\begin{enumerate}[(i)]
\item
$\rho^j_{-1}=0,\ 1\le j\le N;$
\item
$\rho^j_{0}=f_j,\ 1\le j\le N;$
\item
in the double cycle w.r.t. $k=1,2,\ldots,N$, $j=1,2,\ldots, N-k+1$, calculate
\[
\rho^j_{k}=\rho^{j+1}_{k-2}+k/(\rho^{j+1}_{k-1}-\rho^{j}_{k-1}).
\]
We apply Wynn's algorithm with the Gaver functionals
\[
f_j(T)=\frac{j\ln 2}{T}\left(\frac{2j}{j}\right)\sum_{\ell=0}^j (-1)^j\left(\frac{j}{\ell}\right)\tilde f((j+\ell)\ln 2/T).
\]
 \end{enumerate}
 The convergence
of the GS-algorithm is established in \cite{KuznetsovGaverStehfest} but no error estimate is given, and the conditions
for convergence are difficult to verify, for the functions of the complicated nature especially. We suggest
to use the following ad-hoc procedure to estimate the error of the GS and GWR algorithms.
 We take different $a\in\bR$ and modify the Bromwich integral
\bbe\label{tVBrom_a}
V(f;T;x_1,x_2)=\frac{e^{aT}}{2\pi i}\int_{\Re q=\sg}e^{qT}\tV(f;q+a;x_1,x_2)\,dq.
\ee
The difference among the results for $a\in [0,1]$ can be used as a proxy for the error of the GS- or GWR-algorithm.

Note that the simple trick \eq{tVBrom_a} is very useful if $T$ is large which in applications to option pricing means options of long maturities. Then $q=k\ln(2)/T$ is small but efficient calculations of $\tV(f;q,x_1,x_2)$ are possible if $q\ge \sg$, where $\sg>0$ is determined by the parameters of the process and payoff function. Hence, if $T$ is large, we use \eq{tVBrom_a} with $a$ satisfying 
$\ln(2)/T+a>\max\{-\psi(i\mumpr),-\psi(i\muppr)\}$.

\end{document}